\documentclass[a4paper,12pt,reqno]{amsart}
\usepackage{graphicx}

\usepackage{color}

\usepackage{amsmath,amstext,amssymb,amsopn,amsthm}
\usepackage{url}

\usepackage[margin=30mm]{geometry}
\usepackage{eucal,mathrsfs,dsfont}
\usepackage{soul}

\usepackage[cal=boondoxo]{mathalfa}

\newtheorem{theorem}{Theorem}[section]
\newtheorem{corollary}[theorem]{Corollary}
\newtheorem{lemma}[theorem]{Lemma}
\newtheorem{proposition}[theorem]{Proposition}

\theoremstyle{definition}
\newtheorem{conjecture}[theorem]{Conjecture}

\theoremstyle{remark}
\newtheorem{remark}[theorem]{Remark}

\newcommand{\eps}{\varepsilon}

\newcommand{\calB}{\mathcal{B}}
\newcommand{\calF}{\mathcal{F}}
\newcommand{\calG}{\mathcal{G}}
\newcommand{\calK}{\mathcal{K}}
\newcommand{\calL}{\mathcal{L}}
\newcommand{\calN}{\mathcal{N}}
\newcommand{\calR}{\mathcal{R}}

\newcommand{\R}{\mathds{R}}
\newcommand{\Rd}{\mathds{R}^d}
\newcommand{\N}{{\mathds{N}}}

\newcommand{\Bb}{\calB_b(\Rd)}

\newcommand{\tbfs}[1]{\tilde{b}_{#1 1}(x,y)}
\newcommand{\E}{\mathbb{E}}
\newcommand{\p}{\mathbb{P}}

\DeclareMathOperator{\supp}{supp}
\DeclareMathOperator{\dist}{dist}

\title[SDEs driven by L{\'e}vy processes]{Semigroup properties of solutions of SDEs driven by L{\'e}vy processes with independent coordinates}

\author[T. Kulczycki]{Tadeusz Kulczycki}
\author[M. Ryznar]{Micha{\l} Ryznar}
\address{Faculty of Pure and Applied Mathematics, Wroc{\l}aw University of Science and Technology, Wyb. Wyspia{\'n}skiego 27, 50-370 Wroc{\l}aw, Poland.}
\email{Tadeusz.Kulczycki@pwr.edu.pl}
\email{Michal.Ryznar@pwr.edu.pl}

\begin{document}

\begin{abstract} 
We study the stochastic differential equation $dX_t = A(X_{t-}) \, dZ_t$, $ X_0 = x$,
where $Z_t = (Z_t^{(1)},\ldots,Z_t^{(d)})^T$ and $Z_t^{(1)}, \ldots, Z_t^{(d)}$ are independent one-dimensional L{\'e}vy processes with characteristic exponents $\psi_1, \ldots, \psi_d$. We assume that each $\psi_i$ satisfies a weak lower scaling condition WLSC($\alpha,0,\underline{C}$), a weak upper scaling condition WUSC($\beta,1,\overline{C}$) (where $0< \alpha \le \beta < 2$) and some additional regularity properties. We consider two mutually exclusive assumptions: either (i) all $\psi_1, \ldots, \psi_d$ are the same and $\alpha, \beta$ are arbitrary, or (ii) not all $\psi_1, \ldots, \psi_d$ are the same and $\alpha > (2/3)\beta$. We also assume that the determinant of $A(x) = (a_{ij}(x))$ is bounded away from zero, and $a_{ij}(x)$ are bounded and Lipschitz continuous. In both cases (i) and  (ii) we prove that for any fixed $\gamma \in (0,\alpha) \cap (0,1]$ the semigroup $P_t$ of the process $X$ satisfies $|P_t f(x) - P_t f(y)| \le c t^{-\gamma/\alpha} |x - y|^{\gamma} ||f||_\infty$ for arbitrary bounded Borel function $f$. We also show the existence of a transition density of the process $X$.
\end{abstract}

\maketitle

\section{Introduction}
We study the following stochastic differential equation
\begin{equation}
\label{main}
dX_t = A(X_{t-}) \, dZ_t, \quad X_0 = x \in \R^d. 
\end{equation}
We make the following assumptions on a family of matrices $A = (A(x), x \in \R^d)$ and a process $Z = (Z_t, t \ge 0)$.

\vskip 6pt

{\bf{Assumptions (A0).}}
$A(x) = (a_{ij}(x))$ is a $d \times d$ matrix for each $x \in \R^d$ ($d \in \N$, $d \ge 2$). There are constants $\eta_1, \eta_2, \eta_3 > 0$, such that for any $x, y \in \R^d$, $i, j \in \{1,\ldots,d\}$
\begin{equation}
\label{bounded}
|a_{ij}(x)| \le \eta_1,
\end{equation}
\begin{equation}
\label{determinant}
\det(A(x)) \ge \eta_2,
\end{equation}
\begin{equation}
\label{Lipschitz}
|a_{ij}(x) - a_{ij}(y)| \le \eta_3 |x - y|.
\end{equation}
For notational convenience we may and do assume that $\eta_1, \eta_3\ge 1$.
\vskip 6pt

{\bf{Assumptions (Z0).}}
$Z_t = (Z_t^{(1)},\ldots,Z_t^{(d)})^T$, where $Z_t^{(1)}, \ldots, Z_t^{(d)}$ are independent one-dimensional L{\'e}vy processes (not necessarily identically distributed). For each $i \in \{1,\ldots,d\}$  the characteristic exponent $\psi_i$ of the process $Z_t^{(i)}$ is given by
$$
\psi_i(\xi) = \int_{\R} (1 - \cos(\xi x)) \nu_i(x) \, dx,
$$
where $\nu_i(x)$ is the density of a symmetric, infinite L{\'e}vy measure (i.e. $\nu_i:\R\setminus\{0\} \to [0,\infty)$, $\int_{\R} (x^2 \wedge 1) \nu_i(x) \, dx < \infty$, $\int_{\R} \nu_i(x) \, dx = \infty$, $\nu_i(-x) = \nu_i(x)$ for $x \in \R\setminus\{0\}$). There exists $\eta_4 > 0$ such that $\nu_i \in C^1(0,\eta_4)$, $\nu_i'(x) < 0$ for $x \in (0,\eta_4)$ and $-\nu_i'(x)/x$ is decreasing on $(0,\eta_4)$. 
$\psi_i$ satisfies a weak lower scaling condition WLSC($\alpha,0,\underline{C}$) and a weak upper scaling condition WUSC($\beta,1,\overline{C}$) for some constants $0 < \alpha \le \beta < 2$, $\underline{C}, \overline{C} > 0$ (the definitions of WLSC and WUSC are presented in Section 2).

It is well known that under these assumptions SDE (\ref{main}) has a unique strong solution $X$, see e.g. \cite[Theorem 34.7 and Corollary 35.3]{M1982}. By \cite[Corollary 3.3]{SS2010} $X$ is a Feller process.

In the paper we will consider two mutually exclusive assumptions:

\vskip 6pt

{\bf{Assumptions (Z1).}} The process $Z$ satisfies assumptions (Z0). All $\psi_1, \ldots, \psi_d$ are the same.

\vskip 6pt

{\bf{Assumptions (Z2).}} The process $Z$ satisfies assumptions (Z0). Not all $\psi_1, \ldots, \psi_d$ are the same. $\alpha > (2/3) \beta$.

Put $\nu_0(x) = (\nu_1(x),\ldots,\nu_d(x))$. Let $\E^x$ denote the expected value of the process $X$ starting from $x$ and $\Bb$ denote the set of all Borel bounded functions $f: \R^d \to \R$. For any $t \ge 0$, $x \in \R^d$ and $f \in \Bb$ we put 
\begin{equation}
\label{semigroup}
P_t f(x) = \E^x f(X_t).
\end{equation}

The main result of this paper is the following theorem.
\begin{theorem} 
\label{mainthm} Let $A$ satisfy (A0), $Z$ satisfy (Z1) or (Z2), $X$ be the solution of (\ref{main}) and $P_t$ be given by (\ref{semigroup}). Then for any $\gamma \in (0,\alpha) \cap (0,1]$, $\tau > 0$, $t \in (0,\tau]$, $x, y \in \R^d$ and $f \in \Bb$ we have
\begin{equation}
\label{Holder}
|P_t f(x) - P_t f(y)| \le c t^{-\gamma/\alpha} \, |x - y|^{\gamma} \, ||f||_\infty,
\end{equation}
where $c$ depends on $\gamma, \tau, \alpha, \beta, \underline{C}, \overline{C}, d, \eta_1, \eta_2, \eta_3, \eta_4, \nu_0$.
\end{theorem}

This gives the strong Feller property of the semigroup $P_t$. Note that the weaker result namely the strong Feller property of the resolvent $R_s f (x) = \int_0^{\infty} e^{-st} P_t f(x) \, dt$ ($s > 0$) follows from \cite[Theorem 3.6]{SW2012}. Strong Feller property for SDEs driven by additive cylindrical L{\'e}vy processes have been studied recently (see e.g. \cite{PZ2011,DPSZ2016}).

We also show the existence of a transition density of the process $X$.
\begin{proposition} 
\label{heatkernel} 
Let $A$ satisfy (A0), $Z$ satisfy (Z1) or (Z2) and $X$ be the solution of (\ref{main}). Then the  process $X$ has a lower semi-continuous transition density function $p(t,x,y)$, $p: (0,\infty)\times\Rd\times\Rd \to [0,\infty]$ with respect to the Lebesgue measure on $\Rd$.  
\end{proposition}
Recently, the existence of densities for stochastic differential equations driven by L{\'e}vy processes have been studied in \cite{FJR2018} (cf. also \cite{DF2013}). Our existence results and the existence results from \cite{FJR2018} have some intersection. However, their results do not imply ours and our results do not imply theirs. Some more comments on this are in the Remark \ref{existence}.

One may ask about the boundedness of $p(t,x,y)$. It turns out that for some choices of matrices $A$ and processes $Z$ (satisfying assumptions (A0) and (Z1) respectively) and for some $t > 0$ and $x \in \Rd$ we might have $p(t,x,\cdot) \notin L^{\infty}(\R^d)$ (see Remarks 4.23 and 4.24 in \cite{KRS2018}). Nevertheless we have the following regularity result.
\begin{theorem} 
\label{PtL1Linfty} 
Let $A$ satisfy (A0), $Z$ satisfy (Z1) or (Z2), $X$ be the solution of (\ref{main}) and $P_t$ be given by (\ref{semigroup}).
Then for any $\gamma \in (0,\alpha/(d + \beta - \alpha))$, $\tau > 0$, $t \in (0,\tau]$, $x \in \R^d$ and $f \in L^1(\R^d) \cap L^{\infty}(\R^d)$ we have
\begin{equation*}
|P_t f(x)| \le c t^{-\gamma (d + \beta - \alpha)/\alpha} \, \|f\|_\infty^{1-\gamma} \, \|f\|_1^{\gamma},
\end{equation*}
where $c$ depends on $\gamma, \tau, \alpha, \beta, \underline{C}, \overline{C}, d, \eta_1, \eta_2, \eta_3, \eta_4, \nu_0$.
\end{theorem}

Note that we have been able to show only lower semi-continuity of $p(t,x,y)$. In fact, we believe that a stronger result is true.
\begin{conjecture}
Let $A$ satisfy (A0), $Z$ satisfy (Z1) or (Z2) and $X$ be the solution of (\ref{main}). Then the  process $X$ has a continuous transition density function $p(t,x,y)$, $p: (0,\infty)\times\Rd\times\Rd \to [0,\infty]$ with respect to the Lebesgue measure on $\Rd$. If $p(t_0,x_0,y_0) = \infty$ for some $t_0 > 0$, $x_0,y_0 \in \R^d$ then for all $t > 0$, $x \in \R^d$ we have $p(t,x,y_0) = \infty$.  
\end{conjecture}
The continuity should be understood here in the extended sense (as a function with values in $[0,\infty]$).

Estimates of the type $|P_t f(x) - P_t f(y)| \le c_t \, |x - y|^{\gamma} \, ||f||_\infty$ or $|\nabla_x P_tf(x)| \le c_{p,t} \|f\|_{p} $ (for $p > 1$) of semigroups of solutions of SDEs
\begin{equation}
\label{generalSDE}
dX_t = A(X_{t-}) \, dZ_t + b(X_t) \, dt, \quad X_0 = x \in \R^d
\end{equation}
driven by general L{\'e}vy processes $Z$ with jumps have attracted a lot of attention recently. Similarly, of great interest were 
H{\"o}lder or gradient estimates of transition densities of the semigroups of the type $|p(t,x,y) - p(t,z,y)| \le c_{t,y} |x-z|^{\gamma}$, $|\nabla_x p(t,x,y)| \le c_{t,y}$. A lot is known about such estimates when the driving process $Z$ has a non-degenerate diffusion part \cite{XZ2017}. Another well studied case is when $Z$ is a subordinated Brownian motion \cite{WXZ2015}. There are also results for pure-jump L{\'e}vy processes in $\R^d$ such that their L{\'e}vy measure satisfies $\nu(dz) \ge c 1_{|z| \le r} |z|^{-d-\alpha}$ for some $\alpha \in (0,2)$ and $c, r > 0$ \cite{LW2018}. The typical techniques are the coupling method, the use of the Bismut-Elworthy-Li formula or the Levi (parametrix) method. 

Much more demanding case is when the L{\'e}vy measure of the driving process $Z$ is singular. The above gradient type estimates have been studied for SDEs driven by additive cylindrical L{\'e}vy processes (i.e. when $A \equiv I$ and $b \not\equiv 0$ in (\ref{generalSDE})) \cite{WZ2015}. The above H{\"o}lder (or Lipschitz) type estimates for SDEs  driven by processes $Z$ with singular L{\'e}vy measures were also studied in the case when matrices $A(x)$ were diagonal \cite{KR2017}, \cite{LW2018}. The case when the L{\'e}vy measure of the driving process $Z$ is singular and matrices $A(x)$ are not diagonal is much more difficult (heuristically it corresponds to rotations of singular jumping measures). The first important step in understanding this case was done in \cite{KRS2018} in which it was assumed that the driving process $Z$ is a cylindrical $\alpha$-stable process in $\R^d$ with $\alpha \in (0,1)$.
 
The proof of the main result Theorem \ref{mainthm} is based on ideas from \cite{KRS2018}. Similarly as in \cite{KRS2018} we first truncate the L{\'e}vy measure of the process $Z$. Then, as in \cite{KRS2018}, we construct the semigroup of the solution of (\ref{main}), driven by the process with truncated L{\'e}vy measure using the Levi method. Finally, we construct the semigroup of the solution of (\ref{main}), driven by the not truncated process, by (roughly speaking) adding long jumps to the truncated process. 

Nevertheless, there are big differences between this paper and \cite{KRS2018}. First, in \cite{KRS2018} the generators of processes $Z^{(i)}$ are operators of order smaller than $1$ and in this paper they may be of order bigger than $1$. This is much more difficult situation. Secondly, in \cite{KRS2018} the processes $Z_t^{(i)}$ are stable processes and in this paper they are quite general L{\'e}vy processes. The investigation of these processes is much more complicated than stable processes (see Section 2). Thirdly, and most importantly, in \cite{KRS2018} all components $Z_t^{(i)}$ are identically distributed and in our paper we consider the case in which $Z_t^{(i)}$ have different distributions. From technical point of view, in order to use Levi's method, we have to apply generators of $Z_t^{(i)}$ to the density of $Z_t^{(j)}$. When $Z_t^{(i)}$ and $Z_t^{(j)}$ has different distributions this leads to major difficulties in proofs (see e.g. proofs of Lemma \ref{integralgt}, Corollary \ref{int10} and Proposition \ref{integralq02}).

It is worth mentioning that Levi's method has been recently used to study gradient estimates of heat kernels corresponding to various non-local, L{\'e}vy-type operators see e.g. \cite{CZ2016, KSV2018, GS2018}. The coupling method was used in \cite{SSW2012} to obtain gradient estimates of semigroups of transition operators of L{\'e}vy processes satisfying some asymptotic behaviour of their symbols. 
 Let us also add that the properties of harmonic functions corresponding to the solutions of (\ref{main}), when the driving process $Z$ is just the cylindrical $\alpha$-stable process were studied in \cite{BC2010}, (see also \cite{CK2018} for more general results). 

Now we exhibits some examples of processes for which assumptions (Z1) or (Z2) are satisfied. 
\vskip 6pt

{\bf{Example 1.}}
Assume that for each $i \in \{1,\ldots,d\}$ we have $Z_t^{(i)} = B^{(i)}_{S_t^{(i)}}$ where $B_t^{(i)}$ is the one-dimensional Brownian motion and $S_t^{(i)}$ is a subordinator with an infinite L{\'e}vy measure $\mu$ and Laplace exponent $\varphi$. Assume also that $B_t^{(1)},\ldots, B_t^{(d)}, S_t^{(1)},\ldots, S_t^{(d)}$ are independent and for each $i \in \{1,\ldots,d\}$ we have $\varphi \in \text{WLSC}(\alpha/2, 0, \underline{C})$, $\varphi \in \text{WUSC}(\beta/2, 1, \overline{C})$ for some constants $0 < \alpha \le \beta < 2$, $\underline{C}, \overline{C} > 0$. Then assumptions (Z1) are satisfied.

In particular, this holds when $Z_t^{(1)},\ldots,Z_t^{(d)}$ are independent and for each $i \in \{1,\ldots,d\}$ $Z_t^{(i)}$ is a one-dimensional, symmetric $\alpha$-stable process, where $\alpha \in (0,2)$. 

Similarly, this holds when $Z_t^{(1)},\ldots,Z_t^{(d)}$ are independent and for each $i \in \{1,\ldots,d\}$ $Z_t^{(i)}$ is a one-dimensional, relativistic $\alpha$-stable process with $\psi_i(\xi) = \left(m^{2/\alpha} + |\xi|^2\right)^{\alpha/2} - m$, where $\alpha \in (0,2)$, $m > 0$ (cf. \cite{R2001}).

\vskip 6pt

{\bf{Example 2.}}
Assume that for each $i \in \{1,\ldots,d\}$ we have $Z_t^{(i)} = B^{(i)}_{S_t^{(i)}}$ where $B_t^{(i)}$ is the one-dimensional Brownian motion and $S_t^{(i)}$ is a subordinator with an infinite L{\'e}vy measure $\mu_i$ and Laplace exponent $\varphi_i$ such that not all $\varphi_1,\ldots,\varphi_d$ are equal. Assume also that $B_t^{(1)},\ldots, B_t^{(d)}, S_t^{(1)},\ldots, S_t^{(d)}$ are independent and for each $i \in \{1,\ldots,d\}$ we have $\varphi_i \in \text{WLSC}(\alpha/2, 0, \underline{C})$, $\varphi_i \in \text{WUSC}(\beta/2, 1, \overline{C})$ for some constants $0 < \alpha \le \beta < 2$, $\alpha > (2/3) \beta$, $\underline{C}, \overline{C} > 0$. Then assumptions (Z2) are satisfied.

In particular, let $Z_t = (Z_t^{(1)},\ldots,Z_t^{(d)})^T$ be such that $Z_t^{(1)},\ldots,Z_t^{(d)}$ are independent and for each $i \in \{1,\ldots,d\}$ $Z_t^{(i)}$ is a one-dimensional, symmetric $\alpha_i$-stable process ($\alpha_i \in (0,2)$ and they are not all equal). Put $\alpha = \min(\alpha_1,\ldots,\alpha_d)$ and $\beta = \max(\alpha_1,\ldots,\alpha_d)$. If $\alpha > (2/3) \beta$ then assumptions (Q2) are satisfied. The SDE (\ref{main}) driven by such process $Z$ is of great interest see e.g. \cite{C2016}, \cite{C2018}, \cite[example (Z2) on page 2]{FJR2018}.

\vskip 6pt

{\bf{Example 3.}}
Assume that for each $i \in \{1,\ldots,d\}$ the process $Z_t^{(i)}$ is the pure-jump symmetric L{\'e}vy process in $\R$ with the L{\'e}vy measure $\nu(x) \, dx$ given by the formula
\[ \nu(x) = \left\{              
\begin{array}{ll}  \mathcal{A}_{\alpha} |x|^{-1-\alpha}& \text{for} \quad x \in (-1,1)\setminus \{0\},\\
0 & \text{for} \quad |x| > 1,          \end{array}       
\right. \]
where $\mathcal{A}_{\alpha} |x|^{-1-\alpha}$ is the L{\'e}vy density  for the standard one-dimensional, symmetric $\alpha$-stable process, $\alpha \in (0,2)$. Assume also that $Z_t^{(1)},\ldots,Z_t^{(d)}$ are independent. Then assumptions (Z1) are satisfied. Clearly, $Z$ is not a subordinated Brownian motion.

\vskip 6pt

\begin{remark}
\label{existence}
In \cite{FJR2018} the following SDE
\begin{equation*}
dX_t = A(X_{t-}) \, dZ_t + b(X_t) \, dt, \quad X_0 = x \in \R^d. 
\end{equation*}
is studied, where $A(x)$, $b(x)$ are bounded, H{\"o}lder continuous and $Z$ is a L{\'e}vy process in $\R^d$ such that $Z_t$ has a density $f_t$ and there exist $\alpha_1,\ldots,\alpha_d \in (0,2)$ for which we have
$$
\limsup_{t \to 0^+} t^{1/\alpha_k} \int_{\R^d} \left|f_t(z+e_kh) - f_t(z) \, dz\right| \le c |h|, \quad h \in \R,\, k \in \{1,\ldots,d\}.
$$
The main result in \cite{FJR2018} states that there exists a density of $X$ and that the density belongs to the appropriate anisotropic Besov space. This result holds if some conditions on $\alpha_1,\ldots,\alpha_d$ and on the L{\'e}vy measure of $Z$ are satisfied (see \cite[(2.8), (2.9)]{FJR2018}). 

On one hand, the existence result in \cite{FJR2018} holds for some processes $Z$, some matrices $A$ and nonzero drifts $b$ which are not considered in our paper. On the other hand, there are some processes $Z$ for which our result holds and the result in \cite{FJR2018} does not hold, because their conditions on $\alpha_1,\ldots,\alpha_d$ are in some cases more restrictive than our condition $\alpha > (2/3) \beta$. Take for example the process $Z_t = (Z_t^{(1)},\ldots,Z_t^{(d)})^T$ such that $Z_t^{(1)},\ldots,Z_t^{(d)}$ are independent and for each $i \in \{1,\ldots,d\}$ $Z_t^{(i)}$ is a one-dimensional, symmetric $\alpha_i$-stable process ($\alpha_i \in (0,2)$). Put $\alpha =\alpha^{min} = \min(\alpha_1,\ldots,\alpha_d)$ and $\beta = \alpha^{max} = \max(\alpha_1,\ldots,\alpha_d)$. Assume that $\alpha =\alpha^{min} = 1/8$ and $\beta = \alpha^{max} = 1/6$. Then our condition $\alpha/\beta = 3/4 > 2/3$ is satisfied and the condition in \cite[(2.9)]{FJR2018} $\alpha^{min} (1/\gamma + \chi) > 1$ is not satisfied. Indeed, we have
$$
\alpha^{min} \left(\frac{1}{\gamma} + \chi\right) < \alpha^{min} \left(\frac{1}{\alpha^{max}} + 1\right) = \frac{7}{8} < 1.
$$
Note also that we prove that $p(t,x,y)$ is lower semi-continuous in $(t,x,y)$ and no such result is proven in \cite{FJR2018}. Moreover, the methods in \cite{FJR2018} do not give strong Feller property of the semigroup $P_t$.
 \end{remark} 

The paper is organized as follows. In Section 2 we study properties of the transition density of a one-dimensional L{\'e}vy process with a suitably truncated L{\'e}vy measure.  In Section 3 we prove some inequalities involving one dimensional transition densities 
$g_{i,t}^{(\delta)}(x)$ and densities of L{\'e}vy measures $\mu_j^{(\delta)}(w)$ obtained by truncation procedures used in Section 2. In Section 4 we construct the transition density $u(t,x,y)$ of the solution of (\ref{main})  in which the process $Z$ is replaced by a process with a truncated L{\'e}vy measure. We also show that it satisfies the appropriate heat equation in the approximate setting. In Section 5 we construct the transition semigroup of the solution of (\ref{main}). We also prove Theorems \ref{mainthm}, \ref{PtL1Linfty} and Proposition \ref{heatkernel}.

\section{One-dimensional density}\label{1dim}
This section is devoted to showing various estimates of the transition  density and its derivatives for a one-dimensinal   symmetric L\'evy process satisfying certain regularity properties including weak  scaling  conditions. These estimates will play a crucial role in the next sections, specially to make the parametrix construction  in Section \ref{parametrix} work. 

First, we introduce the definition of {\it{ a weak lower scaling condition}} and {\it{ a weak upper scaling condition}} (cf. \cite{BGR2014}). Let $\varphi$ be a non-negative, non-zero function on $[0,\infty)$. We say that $\varphi$ satisfies {\it{a weak lower scaling condition}} \text{WLSC}($\alpha,\theta_1,\underline{C}$) if there are numbers $\alpha > 0$, $\theta_1 \ge 0$ and $\underline{C} > 0$ such that
$$
\varphi(\lambda \theta) \ge \underline{C} \lambda^{\alpha} \varphi(\theta), \quad \text{for} \quad \lambda \ge 1, \, \theta \ge \theta_1.
$$
We say that $\varphi$ satisfies {\it{a weak upper scaling condition}} WUSC($\beta,\theta_2,\overline{C}$) if there are numbers $\beta > 0$, $\theta_2 \ge 0$ and $\overline{C} > 0$ such that
$$
\varphi(\lambda \theta) \le \overline{C} \lambda^{\beta} \varphi(\theta), \quad \text{for} \quad \lambda \ge 1, \, \theta \ge \theta_2.
$$

Let $Z^*$ be a one-dimensional, symmetric L\'evy  process with a characteristic exponent $\psi$ given by
$$
\psi(\xi) = \int_{\R} (1 - \cos(\xi x)) \nu(x) \, dx,
$$
where $\nu(x)$ is the density of a symmetric, infinite L{\'e}vy measure. We assume that there exists $\eta_4 > 0$ such that $\nu \in C^1(0,\eta_4)$, $\nu'(x) < 0$ for $x \in (0,\eta_4)$ and $-\nu'(x)/x$ is decreasing on $(0,\eta_4)$. We also assume that
$\psi$ satisfies a weak lower scaling condition WLSC($\alpha,0,\underline{C}$) and a weak upper scaling condition WUSC($\beta,1,\overline{C}$) for some constants $0 < \alpha \le \beta < 2$, $\underline{C}, \overline{C} > 0$. As a matter of fact we may think that $Z^*$ is any of the processes $Z^{(1)},\ldots,Z^{(d)}$ defined in Introduction. In this section we examine the properties of the transition density of the process $Z^*$  and its truncated version.

Similarly as in \cite{KRS2018} we truncate the density $\nu$  and  the truncated density will be denoted by $\mu^{(\delta)}(x)$. One may easily prove that there exists $\delta_{0} \in (0,1/24]$ such that for any $\delta \in (0,\delta_{0}]$ the following construction of $\mu^{(\delta)}: \R \setminus \{0\} \to [0,\infty)$ is possible. For $x \in (0,\delta]$ we put $\mu^{(\delta)}(x) = \nu(x)$, for $x \in (\delta, 2 \delta)$ we put $\mu^{(\delta)}(x) \in [0,\nu(x)]$ and for $x \ge 2 \delta$ we put $\mu^{(\delta)}(x) = 0$. Moreover, $\mu^{(\delta)}$ is constructed so that $\mu^{(\delta)} \in C^1(0,\infty)$, $(\mu^{(\delta)})'(x) \le 0$ for $x \in (0,\infty)$, $-(\mu^{(\delta)})'(x)/x$ is nonincreasing on $(0,\infty)$ and satisfies $\mu^{(\delta)}(-x) =\mu^{(\delta)}(x)$ for $x \in (0,\infty)$. By $\psi^{(\delta)}$ we denote the characteristic exponent corresponding to the L\'evy measure with density  $\mu^{(\delta)}$.

Let us choose $\delta \in (0,\delta_{0}]$.
We define
$$
\calG^{(\delta)} f(x) = \frac{1}{2} \int_{\R} (f(x+w) + f(x-w) - 2 f(x)) \mu^{(\delta)}(w) \, dw.
$$
By $g_{t}^{(\delta)}$ we denote the heat kernel corresponding to $\calG^{(\delta)}$ that is 
$$
\frac{\partial}{\partial t} g_{t}^{(\delta)}(x) = \calG^{(\delta)} g_{t}^{(\delta)}(x), \quad t > 0, \, x \in \R,
$$
$$
\int_{R} g_{t}^{(\delta)}(x) \, dx = 1, \quad t > 0.
$$
It is well known that $g_{t}^{(\delta)}(x)$ belongs to $C^1((0,\infty))$ as a function of $t$ and  belongs to $C^2(\R)$ as a function of $x$.  

For $r > 0$ we put
$$
h(r) = \int_{\R} (1 \wedge (|x|^2 r^{-2})) \nu(x) \, dx,
$$
$$
h^{(\delta)}(r) = \int_{\R} (1 \wedge (|x|^2 r^{-2})) \mu^{(\delta)}(x) \, dx,
$$
$$
K(r) = \int_{\{x\in\R: \, |x| \le r\}} |x|^2 r^{-2} \nu(x) \, dx,
$$
$$
K^{(\delta)}(r) = \int_{\{x\in\R: \, |x| \le r\}} |x|^2 r^{-2} \mu^{(\delta)}(x) \, dx.
$$
Clearly, $h$ and $h^{(\delta)}$ are decreasing.
By \cite[(6), (7)]{BGR2014} we have
\begin{equation}
\label{eqiv}
\frac 2{\pi^2} h(r)\le  \psi(1/r)\le 2 h(r),\ r>0,\end{equation} 
and the same inequality holds if we replace $\psi$ and $h$
by $\psi^{(\delta)}$ and $h^{(\delta)}$.
By the  scaling properties of $\psi$ and (\ref{eqiv}), if $C_1= \underline{C}/{\pi^2}$ and $C_2=\pi^2\overline{C}$, then 

\begin{equation}
\label{h-scaling_l}
C_1 {\lambda}^{-\alpha} h\left({\theta}\right) \le h\left(\lambda \theta\right),  \   \theta >0,\ 0<\lambda \le 1
\end{equation}
and
\begin{equation}
\label{h-scaling_u}
 h\left(\lambda \theta\right) \le
C_2 {\lambda}^{-\beta} h\left({\theta}\right), \  0<\theta \le 1,\ 0<\lambda \le 1.
\end{equation}
Let us observe that  (\ref{h-scaling_l}) is equivalent to

\begin{equation}
\label{h-scaling_l0}
C_1 {\lambda}^{\alpha}   h\left(\lambda \theta\right)\le h\left({\theta}\right),  \   \theta >0,\ \lambda \ge 1.
\end{equation}
Combining  (\ref{h-scaling_u}) and (\ref{h-scaling_l0}) and taking $\theta=1$ we obtain 

\begin{equation}
\label{h-ubound}
   h\left(x\right)\le (C_2+C_1^{-1}) h\left({1}\right) \left({x}^{-\alpha}+{x}^{-\beta}\right),  \   x >0.
\end{equation}
We also note that the last estimate  together with \cite[ (15)]{BGR2014} and (\ref{eqiv}) yields
\begin{equation}
\label{nu-ubound}
   \nu\left(x\right)\le 16(C_2+C_1^{-1}) h\left({1}\right) \left({x}^{-\alpha-1}+{x}^{-\beta-1}\right),  \   x >0.
\end{equation}

Next, by (\ref{h-scaling_l}) and  \cite[Theorem 1.1 and its proof]{GS2018} we have the following inequality

$$ h(r)\le \left(\frac 2{C_1}\right)^{2/\alpha} K(r),\ r>0.$$

\begin{lemma}\label{K_est}
%

For $r>0$ we have 
$$  K(r)\le  \left(\frac 2{C_1}\right)^{2/\alpha}\frac {r^2\vee\delta^2} {\delta^2} K^{(\delta)}(r)$$
and 
$$h^{(\delta)}(r)\le h(r)\le \left(\frac 2{C_1}\right)^{4/\alpha}\frac {r^2\vee\delta^2}{\delta^2} K^{(\delta)}(r). $$
Also
\begin{equation}
\label{hiKi}
h^{(\delta)}(r)\le 4\left(\frac 2{C_1}\right)^{4/\alpha} K^{(\delta)}(r). 
\end{equation}
\end{lemma}
\begin{proof} 

If $r\le \delta$, then 
$$K(r)= K^{(\delta)}(r)\le  h^{(\delta)}(r)$$


For $r> \delta$, since $ r^2 K^{(\delta)}(r)$ is non-decreasing, we obtain
$$K(r)\le h(r)\le h(\delta)=  \frac{h(\delta)}{\delta^2K(\delta)} \delta^2K^{(\delta)}(\delta)\le 
\left(\frac 2{C_1}\right)^{2/\alpha}  \frac{r^2}{\delta^2}K^{(\delta)}(r). $$
This completes the proof of the first inequality. The second inequality is an obvious consequence of the first one. Finally the last inequality follows from the second one for $ r\le 2\delta$ and for $ r\ge 2\delta$ we have $K^{(\delta)}(r)=h^{(\delta)}(r)$.

\end{proof}

\begin{lemma} 
Let $\tau>0$. For $ t\le \tau$ we have
\begin{equation}
\label{h-1}
C_3 t^{1/\alpha} \le h^{-1}(1/t) \le C_4 t^{1/\beta},
\end{equation}
where $C_3= C_1^{1/\alpha}(h(1)\wedge \frac1\tau)^{1/\alpha}$ and $C_4= C_2^{1/\beta}\left( h^{-1}\left(\frac1\tau\right) \vee 1\right)h\left({1}\right)^{1/\beta} $.
\end{lemma}

\begin{proof}

Taking $\theta=1$ we can rewrite (\ref{h-scaling_l}) and (\ref{h-scaling_u}) as 
$$
C_1^{1/\alpha} h\left({1}\right)^{1/\alpha} h\left({\lambda}\right)^{-1/\alpha} \le \lambda  \le
C_2^{1/\beta} h\left({1}\right)^{1/\beta} h\left({\lambda}\right)^{-1/\beta}, \ 0<\lambda \le 1.$$
Putting  $\lambda= h^{-1}(s)$, for $s\ge h(1)$,  
we have 
$$(C_1 h\left({1}\right))^{1/\alpha} s^{-1/\alpha} \le h^{-1}(s)\le (C_2 h\left({1}\right))^{1/\beta} s^{-1/\beta}.$$
If  $0<s_0\le s\le h(1)$ we have 
$$ s_0^{1/\alpha}s^{-1/\alpha} \le h^{-1}(s) \le h^{-1}(s_0)    h\left({1}\right)^{1/\beta} s^{-1/\beta}.$$
Choosing $\frac1s =t\le \tau $ we show that 
$$
C_1^{1/\alpha}\left(\frac1\tau \wedge h(1)\right)^{1/\alpha} t^{1/\alpha} \le h^{-1}(1/t) \le C_2^{1/\beta} h^{-1}\left(\frac1\tau \wedge h(1)\right)h\left({1}\right)^{1/\beta}  t^{1/\beta}.
$$

\end{proof}

\begin{lemma} 
 For any $x > 0$, $a \in \R \setminus \{0\}$, 
 \begin{equation}
\label{WLSCcaling1}
h\left(\frac{x}{|a|}\right)\le C_1^{-1}(|a|^\alpha + |a|^2) h\left(x\right).
\end{equation}
\end{lemma}

\begin{proof}

By  (\ref{h-scaling_l}), for  $|a|\le 1$, we have
$$h\left(\frac{x}{|a|}\right)\le C_1^{-1}|a|^\alpha  h\left(x\right).$$

Since $x^2h(x)$ is nondecreasing on $( 0,\infty)$, we obtain   for $|a|\ge 1$,  
$$h\left(\frac{x}{|a|}\right)\le   |a|^2 h\left(x\right).$$

Combining both estimates we get the conclusion.\end{proof}

\begin{lemma} \label{momenty}Let $\eta\ge 0$. There is $c=c(\eta, \beta, h(1), C_2)$ such that for all $t>0$,
\begin{equation}
\label{moment}
\int_0^1 x^\eta\left(\frac{1}{h^{-1}\left(\frac{1}{t}\right)} \wedge \frac{t h(x)}{x}\right)dx\le C t^{1\wedge(\eta/\beta)},\ \eta\ne \beta
\end{equation}
and
\begin{equation}
\label{moment1}
\int_0^1 x^\eta\left(\frac{1}{h^{-1}\left(\frac{1}{t}\right)} \wedge \frac{t h(x)}{x}\right)dx\le C t \log(1+1/t),\ \eta= \beta.
\end{equation}

\end{lemma}

\begin{proof} The result in the case $\eta=0$ was proved in \cite{GS2017}, hence we assume that $\eta>0$.

 If $h^{-1}\left(\frac{1}{t}\right)\ge 1$, that is $t\ge \frac1{h(1)}$, we have 

$$\int_0^1 x^\eta\left(\frac{1}{h^{-1}\left(\frac{1}{t}\right)} \wedge \frac{t h(|x|)}{|x|}\right)dx\le \frac{1}{h^{-1}\left(\frac{1}{t}\right)}\le 1,$$
hence the conclusions are true in this case.

Next, we  assume that  $h^{-1}\left(\frac{1}{t}\right)< 1$. Note that, by (\ref{h-1}), 
$$h^{-1}\left(\frac{1}{s}\right)\le  \left(C_2h\left({1}\right)\right)^{1/\beta}s^{1/\beta} $$
 for $s\le 1/h(1)$. 
We have $$I= \int_0^1 x^\eta\left(\frac{1}{h^{-1}\left(\frac{1}{t}\right)} \wedge \frac{t h(x)}{x}\right)dx= 
\int_0^{h^{-1}\left(\frac{1}{t}\right)} + \int_{h^{-1}\left(\frac{1}{t}\right)}^1=I_1+I_2,$$
where
$$I_1= \frac 1{\eta+1} \left(h^{-1}\left(\frac{1}{t}\right)\right)^{\eta}\le  \left(C_2h\left({1}\right)\right)^{\eta/\beta}t^{\eta/\beta}$$
and
$$\eta I_2=  t\int_{h^{-1}\left(\frac{1}{t}\right)}^1 \eta x^{\eta-1} h(x)dx=
t h(1) -\left(h^{-1}\left(\frac{1}{t}\right)\right)^{\eta} - t\int_{h^{-1}\left(\frac{1}{t}\right)}^1  x^{\eta} h^\prime_i(x)dx $$

Next, we estimate the last integral. Let $N$ be the smallest integer such that $h^{-1}\left(\frac{1}{(N+1)t}\right)\ge 1$, then
\begin{eqnarray*}I_3&=& t\int_{h^{-1}\left(\frac{1}{t}\right)}^1  x^{\eta} \left(- h^\prime(x))\right)dx\\
&=& t\sum_{k=1}^N\int_{h^{-1}\left(\frac{1}{kt}\right)}^{h^{-1}\left(\frac{1}{(k+1)t}\right) \wedge 1} x^{\eta} \left(- h^\prime(x))\right)dx\\
&\le& t\sum_{k=1}^N \left[h^{-1}\left(\frac{1}{(k+1)t}\right) \wedge 1\right]^{\eta} \int_{h^{-1}\left(\frac{1}{kt}\right)}^{h^{-1}\left(\frac{1}{(k+1)t}\right)}\left(- h^\prime(x))\right)dx\\
&=&\sum_{k=1}^N \left[h^{-1}\left(\frac{1}{(k+1)t}\right) \wedge 1\right]^{\eta}\left(\frac{1}{k}-\frac{1}{(k+1)}\right)\\
&\le &\left(C_2h\left({1}\right)\right)^{\eta/\beta}\sum_{k=1}^N (t(k+1))^{\eta/\beta}\frac{1}{k^2}
\end{eqnarray*}

Note that $N\le  (h(1)t)^{-1}$, hence the last sum is of order $t^{1\wedge\eta/\beta}$ if $\eta/\beta\ne1$ and of order $t \log (1+1/t)$ if $\eta=\beta$. 
The proof is completed.
\end{proof}

\begin{lemma}\label{g_0} For every $n\in \N$, there is a constant $c=c(n, \alpha, C_1)$ such that 

$$\int_{\R^n}e^{-t\psi^{(\delta)}(|\xi|)} d\xi\le c e^{h(\delta)t} \frac 1{\left(h^{-1}\left(\frac{1}{t}\right)\right)^n}, \quad t>0.$$

\end{lemma}

\begin{proof}

For every $r\in \R$ we have $0\le \psi(r)- \psi^{(\delta)}(r)\le \int_\delta^\infty\nu(u)du\le h(\delta)$, hence using (\ref{eqiv}) we obtain

\begin{eqnarray*}
\int_{\R^n}e^{-t\psi^{(\delta)}(|\xi|)} d\xi 
&\le&  e^{h(\delta)t} \int_{\R^n}e^{-t\psi(|\xi|)} d\xi\\
&\le&  e^{h(\delta)t}  \int_{\R^n}e^{-\frac 2{\pi^2} t h(1/|\xi|)} d\xi\\
&\le& c e^{h(\delta)t} \frac 1{\left(h^{-1}\left(\frac{1}{t}\right)\right)^n},
\end{eqnarray*}
where the last inequality follows from \cite[Lemma 16]{BGR2014}.
\end{proof}

We denote 
$g^*_{t}(x)= \left(\frac{1}{h^{-1}\left(\frac{1}{t}\right)} \wedge  \frac{t h(|x|)}{|x|}\right), t>0, x\in \R$. By Lemma \ref{K_est}, 
according to \cite[Theorem 1.1]{GS2018}, we have the following estimate  

\begin{equation}
\label{g_est0}
\left|\frac{d^k}{dx^k}g_{t}^{(\delta)}(x)\right|\le c [g_{t}^{(\delta)}(0)]^k\left[g_{t}^{(\delta)}(0)\wedge \frac{t h(|x|)}{|x|}\right],\ t> 0,  x\in \R,  \end{equation}
where $c=c(k, \alpha, C_1) $. Since  
$\int_{\R} (\nu(x) - \mu^{(\delta)}(x)) \, dx \le h(\delta)$ we have, by \cite{R2001},
$$\frac {g_{t}^{(\delta)}(0)} {g_{t}(0)}\le e^{h(\delta)t},\ t>0.$$
Hence,
\begin{equation}
\label{g_est01}
\left|\frac{d^k}{dx^k}g_{t}^{(\delta)}(x)\right|\le c e^{(k+1)h(\delta)t} [g_{t}(0)]^k\left[g_{t}(0)\wedge \frac{t h(|x|)}{|x|}\right],\ t> 0, x\in \R.  \end{equation}
Moreover, 
\begin{equation}
\label{g(0)}
g_{t}(0) \le c  \frac 1{h^{-1}\left(\frac{1}{t}\right)},\ t>0,
\end{equation}
where $c=c(\alpha, C_1)$.  



\begin{lemma}
\label{gtht}
For any $\delta \in (0, \delta_0],$ there exist $c=c(\alpha,\delta,h, C_1)$, where $C_1$ is from (\ref{h-scaling_l}),  such that for 
 any $t \in (0,\infty)$, $x\in \R$, we have
\begin{align}
\label{gdelta10}
g_{t}^{(\delta)}(x) 
& \le c (t\wedge1)^{\frac{|x|}{8\delta}} e^{c t} \frac 1{h^{-1}\left(\frac{1}{t}\right)}  e^{-|x|},\\
\label{gdelta12}
\left|\frac{d}{dx}g_{t}^{(\delta)}(x) \right| 
& \le c (t\wedge1)^{\frac{|x|}{8\delta}} e^{c t} \frac 1{h^{-1}\left(\frac{1}{t}\right)^3}  e^{-|x|} ,\\
\label{gdelta13}\left|\frac{d^2}{dx^2}g_{t}^{(\delta)}(x)\right| 
& \le c (t\wedge1)^{\frac{|x|}{8\delta}} e^{c t} \left[\frac 1{h^{-1}\left(\frac{1}{t}\right)^3} +\frac 1{h^{-1}\left(\frac{1}{t}\right)^5}\right] e^{-|x|}\\
\label{gdelta14}
\left|\frac{d}{dx}g_{t}^{(\delta)}(x) \right|
& \le c e^{ct}\frac{|x|}{\left(h^{-1}\left(\frac{1}{t}\right)\right)^2} {g}_{t}^{*}(x).
\end{align}
\end{lemma}
\begin{proof} 
Let $Z^{(\delta)}(t)$ be a L{\'e}vy process in $\R$ with a L{\'e}vy measure $\mu^{(\delta)}(x) \, dx$. Its transition density equals 
$g_{t}^{(\delta)}(x)$. Put $Z_{1}^{(\delta)}(t) := Z^{(\delta)}(t)$, $\mu_{1}^{(\delta)}(x) = \mu^{(\delta)}(x)$, 
$g_{1,t}^{(\delta)}(x) = g_{t}^{(\delta)}(x)$.

By \cite[Theorem 1.5]{KR2016} there exists a L{\'e}vy process 
$Z_{3}^{(\delta)}(t)$ in $\R^3$ with the characteristic exponent $\psi^{(\delta)}(\xi) = \psi^{(\delta)}(|\xi|)$, $\xi \in \R^3$ and the radial, radially nonincreasing transition density 
$g_{3,t}^{(\delta)}(x) = g_{3,t}^{(\delta)}(|x|)$, $x \in \R^3$, satisfying
\begin{equation}
\label{g1tderivative1}
g_{3,t}^{(\delta)}(r) = \frac{-1}{2 \pi r} \frac{d}{dr} g_{1,t}^{(\delta)}(r), \quad r > 0.
\end{equation}
The L{\'e}vy measure of $Z_{3}^{(\delta)}(t)$ has a density $\mu_{3}^{(\delta)}(x) = \mu_{3}^{(\delta)}(|x|)$, $x \in \R^3 \setminus \{0\}$, which satisfies
$$
\mu_{3}^{(\delta)}(r) = \frac{-1}{2 \pi r} \frac{d}{dr} \mu_{1}^{(\delta)}(r), \quad r > 0.
$$
In particular, by our assumptions, $\mu_{3}^{(\delta)}(r)$ is nonincreasing on $(0,\infty)$. Moreover, by monotonicity,
$$\limsup_{t \to 0^+} \frac{1}{t}  g_{3,t}^{(\delta)}(r)\le \mu_{3}^{(\delta)}(r), \quad r > 0,
$$
which implies that
\begin{equation}
\label{g3limit}
\lim_{t \to 0^+} \frac{1}{t}  g_{3,t}^{(\delta)}(r)= 0, \quad r > 2\delta.
\end{equation}

By \cite[Proposition 3.1]{KR2016} there exists a L{\'e}vy process 
$Z_{5}^{(\delta)}(t)$ in $\R^5$ with the characteristic exponent $\psi_{5}^{(\delta)}(\xi) = \psi^{(\delta)}(|\xi|)$, $\xi \in \R^5$,
L{\'e}vy measure $d\mu_{5}^{(\delta)}$ and the radial transition density 
$g_{5,t}^{(\delta)}(x) = g_{5,t}^{(\delta)}(|x|)$, $x \in \R^5$, satisfying
\begin{equation}
\label{g5g3}
g_{5,t}^{(\delta)}(r) = \frac{-1}{2 \pi r} \frac{d}{dr} g_{3,t}^{(\delta)}(r), \quad r > 0.
\end{equation}

We have
\begin{eqnarray}
\nonumber
\frac{d^2}{dr^2} g_{1,t}^{(\delta)}(r)
&=& \frac{d}{dr} \left(- 2 \pi r g_{3,t}^{(\delta)}(r)\right)\\
\nonumber
&=& - 2 \pi g_{3,t}^{(\delta)}(r) - 2 \pi r \frac{d}{dr} g_{3,t}^{(\delta)}(r)\\
\label{gi1tderivative2}
&=& - 2 \pi g_{3,t}^{(\delta)}(r) + (2 \pi r)^2 g_{5,t}^{(\delta)}(r).
\end{eqnarray}

Let $R > 2 \delta^\prime > 2 \delta$. Applying (\ref{g5g3}) and then (\ref{g3limit})  we obtain
\begin{eqnarray*}
&&\int_{B(0,R) \setminus \overline{B(0,2\delta^\prime)}} d\mu_{5}^{(\delta)} (y)
= \lim_{t \to 0^+} \frac{1}{t} \int_{B(0,R) \setminus B(0,2\delta^\prime)} g_{5,t}^{(\delta)}(y) \, dy\\
&& = \frac{8\pi^2}{3}\lim_{t \to 0^+} \frac{1}{t} \int_{2\delta^\prime}^R r^4 g_{5,t}^{(\delta)}(r) \, dr\\
&& = \frac{4\pi}{3}\lim_{t \to 0^+} \frac{1}{t} \int_{2\delta^\prime}^R (-r^3)\frac{d}{dr} g_{3,t}^{(\delta)}(r) \, dr\\
&& \le \frac{4\pi R^3}{3}\limsup_{t \to 0^+} \frac{1 }{t} \int_{2\delta^\prime}^R -\frac{d}{dr} g_{3,t}^{(\delta)}(r) \, dr\\
&& \le \frac{4\pi R^3}{3}\limsup_{t \to 0^+} \frac{1}{t}  g_{3,t}^{(\delta)}(2 \delta^\prime)\\
&& = 0.
\end{eqnarray*}
This gives that $\supp(\mu_{5}^{(\delta)}) \subset \overline{B(0,2\delta)}$.


Denote $d\mu_{n}^{(\delta)}(x) = \mu_{n}^{(\delta)}(x) \, dx$, for $n = 1, 3$.

Let $t\le 1$. Using Lemma 4.2 from \cite{S2017} we get for $n =1, 3, 5$ 
\begin{align*}
  g_{n,t}^{(\delta)}(x) 
	& \leq  e^{\frac{-|x|}{8\delta}\log\left(\frac{\delta |x|}{tm_0}\right)} g_{n,t}^{(\delta)}(0)\\ 
	& =    t^{\frac{|x|}{8\delta}}g_{n,t}^{(\delta)}(0)   e^{\frac{-|x|}{8\delta}\log\left(\frac{\delta |x|}{m_0}\right)}, \quad |x| \geq \tfrac{em_0}{\delta} t,
\end{align*}
where $m_0 =\max \{\int_{\R^n} |y|^2  d\mu_{n}^{(\delta)}(y), n=1,3,5\}$. We observe that there exists $c_1=c_1(\delta,m_0)$ such that
$$
  e^{\frac{-|x|}{8\delta}\log\left(\frac{\delta |x|}{m_0}\right)} \leq c_1 e^{-3|x|}, \quad x\in\R^n.
$$
This yields
$$
  g_{n,t}^{(\delta)}(x) \leq  c_1 t^{\frac{|x|}{8\delta}}g_{n,t}^{(\delta)}(0)e^{-3|x|},
$$
provided $|x| \geq \tfrac{em_0}{\delta} t.$ If $|x|\le \tfrac{em_0}{\delta} t$, then 

$$t^{\frac{|x|}{8\delta}}\ge c_2,$$
where $c_2=c_2(\delta,m_0)$, which implies
$$
  g_{n,t}^{(\delta)}(x) \leq  c^{-1}_2 t^{\frac{|x|}{8\delta}} e^{3\tfrac{em_0}{\delta} t}   g_{n,t}^{(\delta)}(0)e^{-3|x|}.
$$
Hence, for $t<1$, there is a constant $c_3=c_3(\delta, m_0)$ such that 

\begin{equation} \label{g_n}
  g_{n,t}^{(\delta)}(x) \leq  c_3 t^{\frac{|x|}{8\delta}}    g_{n,t}^{(\delta)}(0)e^{-3|x|} , \quad x\in\R^n.
\end{equation}

Let $ t\ge 1$. Using again  Lemma 4.2 from \cite{S2017} we get 
\begin{align}
  g_{n,t}^{(\delta)}(x) 
	& \leq  e^{\frac{-|x|}{8\delta}\log\left(\frac{\delta |x|}{tm_0}\right)} g_{n,t}^{(\delta)}(0) 
	\leq   e^{\frac{-|x|}{8\delta}}g_{n,t}^{(\delta)}(0)\nonumber\\   
	& \leq  g_{n,t}^{(\delta)}(0)   \, e^{-3|x|}, \quad |x| \geq \tfrac{em_0}{\delta} t.
\label{estexpg2}\end{align}
If $|x| \le \tfrac{em_0}{\delta} t$, then 
$$
  g_{n,t}^{(\delta)}(x) \leq   e^{3\tfrac{em_0}{\delta} t}   g_{n,t}^{(\delta)}(0)e^{-3|x|}.
$$
Combining the last two estimates we arrive at 
$$
  g_{n,t}^{(\delta)}(x) \leq   e^{3\tfrac{em_0}{\delta} t}   g_{n,t}^{(\delta)}(0)e^{-3|x|} ,  \quad x\in\R^n, t\ge 1,
$$
which together with (\ref{g_n}) yields
\begin{equation} \label{g_n1}
  g_{n,t}^{(\delta)}(x) \leq  c_3 t^{\frac{|x|}{8\delta}} e^{3\tfrac{em_0}{\delta} t}   g_{n,t}^{(\delta)}(0)e^{-3|x|} , \quad x\in\R^n, t>0,
\end{equation}
where   Next, we note that, by Lemma \ref{g_0} and the inversion Fourier formula, we have 
  $$g_{n,t}^{(\delta)}(0)\le c e^{h(\delta)t} \frac 1{\left(h^{-1}\left(\frac{1}{t}\right)\right)^n},$$
where $c_4=c_4(\alpha,C_1)$. 
This combined with  (\ref{g_n1}), (\ref{g1tderivative1}) and (\ref{gi1tderivative2}) proves the first three inequalities. 

We observe that, by \cite[Theorem 21]{BGR2014}, for $ t> 0, x\in \R^3$ we have
\begin{eqnarray*} g_{3,t}^{(\delta)}(x)&\le& c_5 \left[g_{3,t}^{(\delta)}(0)\wedge \frac{t h(|x|)}{|x|^3}\right]\\
&\le&c_6 e^{h(\delta)t} \frac{1}{\left(h^{-1}\left(\frac{1}{t}\right)\right)^2}
g^*_{t}(|x|),\end{eqnarray*}
where both  $c_5$ and $c_6$ depend only on $C_1$ and $\alpha$.  Hence (\ref{gdelta14}) follows from (\ref{g1tderivative1}).
\end{proof}

For any $\eps \in (0,1], \tau >0 $, $t \in (0,\infty)$ and $x \in \R$ we define
\begin{equation}
\label{hdefinition}
\tilde{g}_{t}^{(\eps)}(x) = \left\{              
\begin{array}{lll}                   
\frac{1}{h^{-1}\left(\frac{1}{t}\right)} \wedge \frac{t h(|x|)}{|x|}&\text{for}& |x| < \eps, \\  
c_{\eps} t^{(d + \beta -1)/\alpha} e^{-|x|} &\text{for}& |x| \ge \eps,
\end{array}       
\right. 
\end{equation}
where $c_{\eps} =  
\left(\frac{1}{h^{-1}\left(\frac{1}{\tau}\right)} \wedge 
\frac{\tau h(|\eps|)}{|\eps|}\right) \frac{e^{\eps}}{\tau^{(d + \beta -1)/\alpha}}$ 
and where we understand $h(0)/0 = \infty$. The constant $c_{\eps}$ is chosen so that for any $t \in (0,\tau]$ the function $x \to \tilde{g}_{t}^{(\eps)}(x)$ is nonincreasing on $[0,\infty)$. Note  that $\tilde{g}_{t}^{(\eps)}$ depends on $\tau$ only by 
$c_{\eps}$.
We observe that for $t \in (0,\tau]$ and $|x| \le h^{-1}(1/t) \wedge \eps$ we have $\tilde{g}_{t}^{(\eps)}(x) = 1/h^{-1}(1/t)$.

The following corollary, whose proof is omitted, follows easily  from the Lemma \ref{momenty} and the definition of $\tilde{g}_{t}^{(\eps)}$.
\begin{corollary}\label{int_tilde_q} For any $0<t\le \tau<\infty$, 
$$\int_{\R}\tilde{g}_{t}^{(\eps)}(x)dx \le c,$$
where $c=c(\eps, h, \tau)$.
Moreover, for any $a>0$,

$$\lim_{t\to 0^+}\int_{|x|>a}\tilde{g}_{t}^{(\eps)}(x)dx =0.$$
\end{corollary}

We introduce the following convention. For a function $f$ and arguments $x,u \in \R$ we write $f(x\pm u)= f(x- u)+f(x+ u)$.

\begin{lemma}
\label{gtht1}
For any $\eps \in (0,1], \tau>0,$ there exists $c$ such that for 
$$\delta = \min\{\delta_0, \eps \alpha /(8d+8\beta + 16), \eps/(d \eta_1^2)\},$$ and any $t \in (0,\tau]$, $x,u,w\in \R$, we have
\begin{align}
\label{gdelta1}
g_{t}^{(\delta)}(x) 
& \le c \tilde{g}_{t}^{(\eps)}(x),\\
\label{gdelta2}
|g_{t}^{(\delta)}(x+u) - g_{t}^{(\delta)}(x)| 
& \le \frac{c |u|}{h^{-1}\left(\frac{1}{t}\right)} (\tilde{g}_{t}^{(\eps)}(x+u) + \tilde{g}_{t}^{(\eps)}(x)),\\
\label{gdelta3}
|g_{t}^{(\delta)}(x\pm u) - 2g_{t}^{(\delta)}(x)| 
&\le \frac{c |u|^2}{\left(h^{-1}\left(\frac{1}{t}\right)\right)^2} \max_{\xi \in [x-|u|,x+|u|]} \tilde{g}_{t}^{(\eps)}(\xi),\\
\label{gdelta4}
|g_{t}^{(\delta)}(x\pm u)-g_{t}^{(\delta)}(x\pm w)|
\le& \frac{c \left||u|^2 -|v|^2\right|}{\left(h^{-1}\left(\frac{1}{t}\right)\right)^2}
(g^*_{t}(x\pm u)+ g^*_{t}(x\pm w)), 
\end{align}
where $c= c(\alpha, \beta, \delta, \tau, h, C_1, C_2)$.
\end{lemma}
\begin{proof} Here in the proof below a constant $c$ may change its value from line to line but it is dependent only  on
$\alpha, \beta, \delta, \tau, h, C_1, C_2$. 
To prove (\ref{gdelta1} - \ref{gdelta3}) it is enough to show that, for $k=0,1,2$ and $x\in\R$, we have

\begin{equation}\label{der_est}\left|\frac{d^k}{dx^k}g_{t}^{(\delta)}(x) \right|\le c \frac 1{\left(h^{-1}\left(\frac{1}{t}\right)\right)^k}\tilde{g}_{t}^{(\eps)}(x), \  t\le \tau. \end{equation}

For $|x|\le\eps$,  (\ref{der_est}) follows directly from (\ref{g_est01}) and  (\ref{g(0)}).
From (\ref{h-1}) we infer that  $\frac 1{h^{-1}(1/t) }\le c t^{-1/\alpha}$, hence applying (\ref{gdelta10}) we obtain

$$g_{t}^{(\delta)}(x) 
 \le c (t\wedge1)^{\frac{|x|}{8\delta}}  \frac 1{h^{-1}\left(\frac{1}{t}\right)}  e^{-|x|}\le c t^{-1/\alpha}(t\wedge1)^{\frac{|x|}{8\delta}}   e^{-|x|}, \ t\le \tau.$$
Similarly, by (\ref{gdelta12}) and (\ref{gdelta13}), we have
$$\left|\frac{d}{dx}g_{t}^{(\delta)}(x) \right| 
\le c t^{-2/\alpha}(t\wedge1)^{\frac{|x|}{8\delta}} \frac 1{h^{-1}\left(\frac{1}{t}\right)}   e^{-|x|}, \ t\le \tau $$
and
$$\left|\frac{d^2}{dx^2}g_{t}^{(\delta)}(x)\right| 
\le c t^{-3/\alpha}(t\wedge1)^{\frac{|x|}{8\delta}} \frac 1{\left(h^{-1}\left(\frac{1}{t}\right)\right)^2}   e^{-|x|}, \ t\le \tau. $$
By the choice of $\delta$, for $|x|\ge\eps$ we have  $\frac{|x|}{8\delta} \ge \frac{d+\beta+2}\alpha$, which proves 
(\ref{der_est}) if  $|x|\ge\eps$.

Next,   for $ 0\le a\le b$ we have, by (\ref{gdelta14}),
\begin{eqnarray}|g_{t}^{(\delta)}(b)-g_{t}^{(\delta)}(a)|
&=& \left|\int_a^{b} \frac d{d \xi} g_{t}^{(\delta)}(\xi) d\xi\right| \nonumber\\
&\le& c e^{ct}\frac{1}{\left(h^{-1}\left(\frac{1}{t}\right)\right)^2} 
 \int_a^{b}\xi {g}_{t}^{*}(\xi) d\xi \nonumber\\
&\le & 
c e^{ct}\frac{1}{\left(h^{-1}\left(\frac{1}{t}\right)\right)^2}{g}_{t}^{*}(a) \int_a^{b}\xi  d\xi
\nonumber\\
&=& \frac12c e^{ct}\frac{1}{\left(h^{-1}\left(\frac{1}{t}\right)\right)^2}{g}_{t}^{*}(a) \left(b^2 -a^2\right).\label{diff}
\end{eqnarray}
   %

Now, we proceed with the proof of (\ref{gdelta4}). It is enough to prove it for $x>0$ and $0\le u\le w$.

We first consider the case  $|x-u|> |x-w|$, which can be split into two subcases. One of them  is the subcase 
  $0\le u \le  w\le x$. Then, by (\ref{g_est01}) and (\ref{g(0)}), we obtain
\begin{eqnarray}&&|g_{t}^{(\delta)}(x+u)+ g_{t}^{(\delta)}(x-u)-(g_{t}^{(\delta)}(x+w)+ g_{t}^{(\delta)}(x-w))|\nonumber\\&=&\left| \int_u^{w}\left( \frac d{d \xi} g_{t}^{(\delta)}(x+\xi)-\frac d{d \xi} g_{t}^{(\delta)}(x-\xi)\right)d\xi\right|\nonumber\\
&\le& \sup_{x-w\le\xi\le x+w} \left|\frac {d^2}{d \xi^2} g_{t}^{(\delta)}(\xi)\right|(w^2-u^2)
 \nonumber\\
&\le & c e^{ct} \frac{1}{\left(h^{-1}\left(\frac{1}{t}\right)\right)^2}
g^*_{t}(x-w)(w^2-u^2) \label{diff1}.
\end{eqnarray}
Next, consider the second subcase  $u\le x\le w $ and $|x-w|< x-u$. Denote $w^*= 2x-w$, then  $0<u\le w^*\le x$ and $|x- w^*|= |x- w|$. Hence, by (\ref{diff}) and  (\ref{diff1}),  
\begin{eqnarray}&&|g_{t}^{(\delta)}(x+u)+ g_{t}^{(\delta)}(x-u)-(g_{t}^{(\delta)}(x+w)+ g_{t}^{(\delta)}(x-w))|\nonumber \\
&&\le |g_{t}^{(\delta)}(x+w)-g_{t}^{(\delta)}(x+w^*)|\nonumber\\
 &&+|g_{t}^{(\delta)}(x+w^*) +g_{t}^{(\delta)}(x-w^*) - (g_{t}^{(\delta)}(x+u) +g_{t}^{(\delta)}(x-u) )|\nonumber\\
&&\le  c e^{ct} \frac{1}{\left(h^{-1}\left(\frac{1}{t}\right)\right)^2}
g^*_{t}(x-w)(w^2-{w^*}^2+{w^*}^2-u^2).\label{diff2}\end{eqnarray}
Combining (\ref{diff1}) and (\ref{diff2}) we get (\ref{gdelta4}) if $|x-w|< |x-u|$.
Next, we consider the case $|x-w|\ge |x-u|$. We have 
\begin{eqnarray*}&&|g_{t}^{(\delta)}(x\pm u)-(g_{t}^{(\delta)}(x\pm w)| \\
&=&|g_{t}^{(\delta)}(x+u)+ g_{t}^{(\delta)}(|x-u|)-(g_{t}^{(\delta)}(x+w)+ g_{t}^{(\delta)}(|x-w|))| \\
&\le&|g_{t}^{(\delta)}(x+u)-g_{t}^{(\delta)}(x+w)|+  |g_{t}^{(\delta)}(|x-u|)- g_{t}^{(\delta)}(|x-w|)|
\end{eqnarray*}
By (\ref{diff}) this is bounded form above by
\begin{eqnarray*}
&&   \frac{c e^{ct}}{\left(h^{-1}\left(\frac{1}{t}\right)\right)^2}
\left[g^*_{t}(x+u)((x+w)^2-(x+u)^2) + g^*_{t}(|x-u|)((x-w)^2-(x-u)^2)\right]\\
&\le&   c e^{ct}\frac{1}{\left(h^{-1}\left(\frac{1}{t}\right)\right)^2}
g^*_{t}(|x-u|) ((x+w)^2-(x+u)^2+(x-w)^2-(x-u)^2)\\&=&   2c e^{ct}\frac{1}{\left(h^{-1}\left(\frac{1}{t}\right)\right)^2}
g^*_{t}(x-u)(w^2-u^2),\end{eqnarray*}
which completes the proof of (\ref{gdelta4}).
\end{proof}

\begin{lemma}
\label{xxprime}
Let $\eps \in (0,1]$. For any $t \in (0,\tau]$, 
$x, x' \in \R$ if $|x- x'| \le h^{-1}(1/t)/4$ and $|x - x'| \le \eps/4$ then
$$
\tilde{g}_{t}^{(\eps)}(x') \le  \tilde{g}_{t}^{(\eps)}(x/2).
$$
\end{lemma}
\begin{proof}
Recall that $x \to \tilde{g}_{t}^{(\eps)}(x)$ is nonincreasing on $[0,\infty)$ and $\tilde{g}_{t}^{(\eps)}(-x) = \tilde{g}_{t}^{(\eps)}(x)$ for $x \in \R$. Therefore we may assume that $x > x' \ge 0$.

Assume that $x' \le (h^{-1}(1/t)/2) \wedge (\eps/2)$. Then $x = x' + x - x' \le h^{-1}(1/t) \wedge \eps$ so 
$\tilde{g}_{t}^{(\eps)}(x') =  \tilde{g}_{t}^{(\eps)}(x) = 1/(h^{-1}(1/t))$.

Assume now that $x' > (h^{-1}(1/t)/2) \wedge (\eps/2)$. Then we have $x' = x - (x - x') \ge x - x/2 = x/2$. 
Hence $\tilde{g}_{t}^{(\eps)}(x') \le  \tilde{g}_{t}^{(\eps)}(x/2)$.
\end{proof}

\section{Some useful estimates}


In this section we prove several inequalities used in the sequel involving some relationships between one dimensional densities and L\'evy measures of processes obtained by appropriate truncation procedures described in Section \ref{1dim}. From now on our basic assumption on the process $Z_t = (Z_t^{(1)},\ldots,Z_t^{(d)})^T$ is the assumption (Z0). Let us recall that by $\nu_i$ we denote the density of the L\'evy measure of the process $Z_t^{(i)}, i \in \{1,\dots, d\}$. 

Let $\tau>0$ and $\eps, \delta\in (0, 1]$, where usually $\delta$ is picked conveniently.      For each $i \in \{1,\dots, d\}$ we denote by 
$h_i,   \mu_i^{(\delta)}, {g}_{i,t}^{(\delta)}, \tilde{g}_{i,t}^{(\eps)}$ all the objects defined  in Section \ref{1dim} but now corresponding to the measure $\nu_i$.
Under our  assumptions we may pick a positive  $\delta_0\le 1/24$ such that for all $\delta \in (0,\delta_0]$ the truncated Levy measures $\mu_i^{(\delta)}$ have the properties required in Section \ref{1dim}, hence we can apply all the estimates proved in that section. 

We adopt the convention that constants denoted by $c$ (or $c_1, c_2, \ldots$) may change their value from one use to the next. In the rest of the paper, unless is explicitly stated otherwise, we understand that constants denoted by $c$ (or $c_1, c_2, \ldots$) depend on ${\bf{ \nu_0}}, \tau, \alpha, \beta, \underline{C}, \overline{C}, d, \eta_1, \eta_2, \eta_3, \eta_4$. We  also understand that they may depend on the choice of the constants $\delta_0$, $\eps_0$ and $\gamma$. We write $f(x) \approx g(x)$ for $x \in A$ if $f, g \ge 0$ on $A$ and there is a constant $c \ge 1$ such that $c^{-1} f(x) \le g(x) \le c f(x)$ for $x \in A$. The standard inner product for $x, y \in \R^d$ we denote by $xy$. We denote by $B(x,r)$ an open  ball of the center $x\in \R^d$ and radius $r>0$.


\begin{lemma}
\label{product_gt}
Let $\eps \in (0,1]$, $\delta = \min\{\delta_0 , \frac{\eps \alpha}{8(d+\beta + 2)},\frac{\eps}{d\eta_1^2}\}$. For any 
$t \in (0,\tau]$, $x, x' \in \R^d$ if $\sum_{j = 1}^d \frac{|x_j - x'_j|}{h_j^{-1}(1/t)} \le \frac{1}{4}$ and $|x - x'| \le \delta$ then
\begin{equation}
\label{product_gt1}
\left| \prod_{i=1}^d  g_{i,t}^{(\delta)}\left({x_i}\right) - 
\prod_{i=1}^d  g_{i,t}^{(\delta)}\left({x'_i}\right)\right|
\le c  \left(\prod_{i=1}^d \tilde{g}_{i,t}^{(\eps)}(x_i/2) \right)\left[1\wedge \sum_{j = 1}^d \frac{|x_j - x'_j|}{h_j^{-1}(1/t)}\right].
\end{equation}
\end{lemma}
\begin{proof}
By Lemma \ref{gtht1} we get
\begin{eqnarray*}
&&
\left| \prod_{i=1}^d  g_{i,t}^{(\delta)}\left({x_i}\right) - 
\prod_{i=1}^d  g_{i,t}^{(\delta)}\left({x'_i}\right)\right|\\
&&
\le  \sum_{j=1}^d \left[\left|g_{j,t}^{(\delta)}\left(x_{j}\right) - g_{j,t}^{(\delta)}\left(x'_{j}\right)\right| 
\prod_{i\ne j, 1\le i\le d} g_{i,t}^{(\delta)}\left(|x_{i}|\wedge |x'_{i}|\right) \right]\\
&& \le c \left(\prod_{i=1}^d \tilde{g}_{i,t}^{(\eps)}(|x_i|\wedge |x'_i|) \right) \sum_{j = 1}^d  
\frac{|x_j-x'_j|}{h_j^{-1}(1/t)}.
\end{eqnarray*}

Clearly we have
$$\left| \prod_{i=1}^d  g_{i,t}^{(\delta)}\left({x_i}\right) - 
\prod_{i=1}^d  g_{i,t}^{(\delta)}\left({x'_i}\right)\right|
\le \prod_{i=1}^d g_{i,t}^{(\delta)}(|x_i|\wedge |x'_i|).$$

Now the assertion follows from Lemma \ref{xxprime}.
\end{proof}

\begin{lemma}
\label{integralgt}
Let $\eps \in (0,1]$, $\delta = \min\{\delta_0 , \frac{\eps \alpha}{8(d+\beta + 2)},\frac{\eps}{d\eta_1^2}\}$ and let $a, b\in\R$. Then there exists $c$ such that for any $t \in (0,\tau]$, $x \in \R$, $i, j, k \in \{1,\ldots,d\}$ we have

\begin{eqnarray}\nonumber
&& \int_{\R} |g_{i,t}^{(\delta)}(x + aw)+g_{i,t}^{(\delta)}(x - aw) - 2g_{i,t}^{(\delta)}(x)| \mu_j^{(\delta)}(w) \, dw \\
\label{gdelta41}
&\le& \frac{c (|a|^{\alpha}+ |a|^{2})\left(\max_{|u|\le 2\delta }
 \tilde{g}_{i,t}^{(\eps)}(x + au )\right)}{t^{\beta/\alpha}},
\end{eqnarray}

\begin{eqnarray}
\nonumber
&&\int_{\R} |g_{i,t}^{(\delta)}(x + aw)+g_{i,t}^{(\delta)}(x - aw) - 2g_{i,t}^{(\delta)}(x)| \mu_i^{(\delta)}(w) \, dw \\
\label{gdelta31}
&\le& \frac{c (|a|^{\alpha}+ |a|^{2})\left(\max_{|u|\le 2\delta }
 \tilde{g}_{i,t}^{(\eps)}(x + au )\right)}{t},
\end{eqnarray}

\begin{eqnarray}
&\int_{\R} |g_{i,t}^{(\delta)}(x + aw)-g_{i,t}^{(\delta)}(x)| 
|g_{k,t}^{(\delta)}(y + bw)-g_{k,t}^{(\delta)}(y)|\mu_j^{(\delta)}(w) \, dw \nonumber\\ 
\label{gdelta42}
&\le c  \frac{(|a||b|)^{\alpha/2} +|a||b|}{t^{\beta/\alpha }}\left(\max_{|u|\le 2\delta }
 \tilde{g}_{i,t}^{(\eps)}(x + au )\right)\left(\max_{|u|\le 2\delta }
 \tilde{g}_{k,t}^{(\eps)}(y + bu ) \right), 
\end{eqnarray}

\begin{eqnarray} 
&\int_{\R} |g_{i,t}^{(\delta)}(x + aw)-g_{i,t}^{(\delta)}(x)| 
|g_{i,t}^{(\delta)}(y + bw)-g_{i,t}^{(\delta)}(y)|\mu_i^{(\delta)}(w) \, dw \nonumber\\ 
\label{gdelta32}
&\le c  \frac{(|a||b|)^{\alpha/2} +|a||b|}{t}\left(\max_{|u|\le 2\delta }
 \tilde{g}_{i,t}^{(\eps)}(x + au )\right)\left(\max_{|u|\le 2\delta }
 \tilde{g}_{i,t}^{(\eps)}(y + bu ) \right). 
\end{eqnarray}

\end{lemma}
\begin{proof}
Let $|w|\le 2\delta$. Then, by (\ref{gdelta1}) and (\ref{gdelta3}), 
$$
|g_{i,t}^{(\delta)}(x+aw)+ g_{i,t}^{(\delta)}(x-aw) - 2g_{i,t}^{(\delta)}(x)| 
\le c\left(\frac{ |a|^2|w|^2}{\left(h_i^{-1}\left(\frac{1}{t}\right)\right)^2}\wedge1\right)\max_{|u|\le 2\delta }
 \tilde{g}_{i,t}^{(\eps)}(x + au ).
$$

First we show (\ref{gdelta41}) and (\ref{gdelta31}).
We have
\begin{eqnarray*}
	&& \int_{\R} 
	|g_{i,t}^{(\delta)}(x + aw)+g_{i,t}^{(\delta)}(x - aw) - 2g_{i,t}^{(\delta)}(x)| 
	\mu_j^{(\delta)}(w) \, dw\\
	&& \leq \int_{|w|<2\delta} 
	|g_{i,t}^{(\delta)}(x + aw)+g_{i,t}^{(\delta)}(x - aw) - 2g_{i,t}^{(\delta)}(x)| 
	\nu_j(w) \, dw\\
	&& \leq c\max_{|u|\le 2\delta }
 \tilde{g}_{i,t}^{(\eps)}(x + au ) \int_{\R} 
	\left(\frac{ |a|^2|w|^2}{\left(h_i^{-1}\left(\frac{1}{t}\right)\right)^2}\wedge1\right)
	\nu_j(w) \, dw\\
	&& = c\max_{|u|\le 2\delta }
 \tilde{g}_{i,t}^{(\eps)}(x + au )h_j\left(\frac {h_i^{-1}\left(\frac{1}{t}\right)}{|a|}\right)
\end{eqnarray*}
Next, by (\ref{WLSCcaling1}),
$$h_j\left(\frac {h_i^{-1}\left(\frac{1}{t}\right)}{|a|}\right)\le c(|a|^\alpha + |a|^2) h_j\left( {h_i^{-1}\left(\frac{1}{t}\right)}\right).$$
By (\ref{h-scaling_u}) and (\ref{h-1}) we get
$${h_j\left({h_i^{-1}\left(\frac{1}{t}\right)}\right)}\le c t ^{-\beta/\alpha}$$ if $i\ne j$ and 
$${h_i\left({h_i^{-1}\left(\frac{1}{t}\right)}\right)}=  t ^{-1}.$$
This finishes the proof of  (\ref{gdelta41}) and (\ref{gdelta31}).

 To show (\ref{gdelta42}) and (\ref{gdelta32})  we 
use (\ref{gdelta1}) and (\ref{gdelta2}) to   obtain
\begin{eqnarray}
\nonumber
&& |g_{i,t}^{(\delta)}(x + aw)-g_{i,t}^{(\delta)}(x)| 
|g_{k,t}^{(\delta)}(y + bw)-g_{k,t}^{(\delta)}(y)|\\
\nonumber
&& \le c\left(\frac{ |a| |b| w^2}{h_i^{-1}(1/t) h_k^{-1}(1/t)}\wedge1\right)
\max_{|w|\le 2\delta }
 \tilde{g}_{i,t}^{(\eps)}(y + aw ) \max_{|w|\le 2\delta }
 \tilde{g}_{k,t}^{(\eps)}(y + bw ) 
\end{eqnarray}
Therefore 
\begin{eqnarray*}
&& \int_{\R} |g_{i,t}^{(\delta)}(x + aw)-g_{i,t}^{(\delta)}(x)| 
|g_{k,t}^{(\delta)}(y + bw)-g_{k,t}^{(\delta)}(y)|\mu_j^{(\delta)}(w) \, dw\\
&& \le c  h_j\left(\sqrt{\frac{h_i^{-1}(1/t)h_k^{-1}(1/t)}{|a| |b|}}\right)  \max_{|u|\le 2\delta }
 \tilde{g}_{i,t}^{(\eps)}(x + au ) \max_{|u|\le 2\delta }
 \tilde{g}_{k,t}^{(\eps)}(y + bu )   \\
&& \le  c \left[{(|a||b|)^{\alpha/2}+|a||b|}\right]h_j\left(\sqrt{{h_i^{-1}(1/t)h_k^{-1}(1/t)}}\right) \max_{|u|\le 2\delta }
 \tilde{g}_{i,t}^{(\eps)}(x + au ) \max_{|u|\le 2\delta }
 \tilde{g}_{k,t}^{(\eps)}(y + bu ) 
\end{eqnarray*}
which proves (\ref{gdelta42}) and (\ref{gdelta32}) since 
$h_j\left(\sqrt{{h_i^{-1}(1/t)h_k^{-1}(1/t)}}\right)$ is equal to $t^{-1}$ if $i=j=k$ and it is smaller than  $ c t^{-\beta/\alpha}$ in the general case.

\end{proof}

\begin{lemma}
\label{levy1} There is a constant $c$ such that for $ a \in \R$  any $0<t\le\tau$ and $i, j, k \in \{1, \ldots, d\}$ we have 
\begin{equation}
\label{levy1a}
\int_{\R} \left(\frac{(|a|+|w|)^2|w|^2}{\left(h_i^{-1}(1/t)\right)^2}\wedge1\right)\mu_i^{(\delta)}(w) \, dw\le c 
\frac{t^{\alpha/(2\beta)}+|a|^{\alpha}+|a|^{2}}{t},
\end{equation}
\begin{equation}
\label{levy1b}
\int_{\R} \left(\frac{(|a|+|w|)|w|^2}{\left(h_i^{-1}(1/t)\right)^2}\wedge1\right)\mu_i^{(\delta)}(w) \, dw\le c 
\frac{t^{\alpha/(3\beta)}+|a|^{\alpha/2}+|a|}{t},
\end{equation}

\begin{equation}
\label{levy2a}
\int_{\R} \left(\frac{(|a|+|w|)^2|w|^2}{h_i^{-1}(1/t)h_j^{-1}(1/t)}\wedge1\right)\mu_k^{(\delta)}(w) \, dw\le c 
\left[|a|^{\beta} t^{-\beta/\alpha}+   t^{-{\beta}/(2\alpha)}\right],
\end{equation}
\begin{equation}
\label{levy2b}
\int_{\R} \left(\frac{(|a|+|w|)|w|^2}{\left(h_i^{-1}(1/t)\right)^2}\wedge1\right)\mu_k^{(\delta)}(w) \, dw\le c \left[
|a|^{\beta/2} t^{-\beta/\alpha}+ c_2  {t^{-{2\beta}/(3\alpha)}}\right].
\end{equation}
\end{lemma}
\begin{proof} Let $k\ge 1$, $k \in \N$ and $b>0$. Then

\begin{eqnarray}
 \int_{\R} \left(\frac{(|a|+|w|)^k|w|^2}{b}\wedge1\right)\mu_i^{(\delta)}(w) \, dw  \nonumber
& \le&2^{k-1} \int_{\R} \left(\frac{|a|^k|w|^2}{b}\wedge1\right)\nu_i(w) \, dw\\ \nonumber
 &+& 2^{k-1} \int_{\R} \left(\frac{|w|^{k+2}}{b}\wedge1\right)\nu_i(w) \, dw\\ \nonumber
& \le& 2^{k-1} \int_{\R} \left(\frac{|a|^k|w|^2}{b}\wedge1\right)\nu_i(w) \, dw\\ \nonumber
&+& 2^{k-1} \int_{\R} \left(\frac{|w|^{2}}{b^{2/(k+2)}}\wedge1\right)\nu_i(w) \, dw\\
& =& 2^{k-1}\left(h_i\left(\frac{\sqrt{b}}{|a|^{k/2}}\right)+ h_i\left(b^{\tfrac1{k+2}}\right)\right).\label{levy3}
\end{eqnarray}
Taking $b= (h_i^{-1}(1/t))^2$ and $k=2$ we arrive at 
$$
\int_{\R} \left(\frac{(|a|+|w|)^2|w|^2}{\left(h_i^{-1}(1/t)\right)^2}\wedge1\right)\mu_i^{(\delta)}(w) \, dw
\le c \left(h_i\left(\frac{h_i^{-1}(1/t)}{|a|}\right)+ h_i\left(\sqrt{h_i^{-1}(1/t)}\right)\right).
$$
Next, by the scaling property (\ref{WLSCcaling1}), $h_i\left(\frac{h_i^{-1}(1/t)}{|a|}\right)\le c \frac {|a|^{\alpha}+a^2}t$. Moreover,
by (\ref{WLSCcaling1}) and (\ref{h-1}), we get
$$
 h_i\left(\frac{1}{\left(h_i^{-1}(1/t)\right)^{1/2} } h_i^{-1}(1/t)\right)
\le c \frac{\left(h_i^{-1}(1/t)\right)^{\alpha/2}}t\le ct^{\alpha/(2\beta)}/t.
$$
The proof of (\ref{levy1a}) is completed.

By similar arguments, taking $b= (h_i^{-1}(1/t))^2$ and $k=1$ in (\ref{levy3}),  we arrive at (\ref{levy1b}).

Now, we proceed with the proof of (\ref{levy2a}) and (\ref{levy2b}).
First, we observe that, by  (\ref{h-1}),\\  $h_i^{-1}(1/t)h_j^{-1}(1/t)\ge c t^{2/\alpha}$ and  $ \mu_k^{(\delta)}(w)\le \frac c{|w|^{1+\beta}}$ , by (\ref{nu-ubound}). Hence,

\begin{eqnarray*}\int_{\R} \left(\frac{(|a|+|w|)^2|w|^2}{h_i^{-1}(1/t)h_j^{-1}(1/t)}\wedge1\right)\mu_k^{(\delta)}(w)&\le& c
\int_{\R} \left(\frac{(|a|+|w|)^2|w|^2}{t^{2/\alpha}}\wedge1\right)\frac 1{|w|^{1+\beta}}\, dw\\
&\le&c\int_{\R} \left(\frac{|a|^2|w|^2}{t^{2/\alpha}}\wedge1\right)\frac 1{|w|^{1+\beta}}\, dw\\
&+&c\int_{\R} \left(\frac{|w|^4}{t^{2/\alpha}}\wedge1\right)\frac 1{|w|^{1+\beta}}\, dw\\
&=&c_1 |a|^{\beta} t^{-\beta/\alpha}+ c_2  {t^{-{\beta}/(2\alpha)}}.
\end{eqnarray*}
Similar calculations show that
\begin{eqnarray*}\int_{\R} \left(\frac{(|a|+|w|)|w|^2}{h_i^{-1}(1/t)h_j^{-1}(1/t)}\wedge1\right)\mu_k^{(\delta)}(w)&\le&
c |a|^{\beta/2} t^{-\beta/\alpha}+ c_2  {t^{-{2\beta}/(3\alpha)}}.
\end{eqnarray*}
The proof is completed.
\end{proof}

\section{Construction and properties of the transition density of the solution of (\ref{main}) driven by the truncated process}\label{parametrix}

The approach in this section is based on Levi's method (cf. \cite{L1907,F1975,LSU1968}). This method was applied  in the framework of pseudodifferential operators by Kochubei \cite{K1989} to construct a fundamental solution to the related Cauchy problem as well as  transition density for the  corresponding Markow process. In recent years it was used in several papers to study transition densities of L{\'e}vy-type processes see e.g. \cite{CZ2016, KSV2018, CZ2017, J2017, GS2018, BSK2017, KK2017, KK2018, K2017}. Levi's method was also used to study gradient and Schr{\"o}dinger perturbations of fractional Laplacians see e.g. \cite{BJ2007,CKS2012,XZ2014}.

From now on we  assume that the assumptions (A0), and either (Z1) or (Z2) are satisfied.
We first introduce the generator of the process $X$. We define $\calK f(x)$ by the following formula
$$
\calK f(x) =  \frac{1}{2} \sum_{i = 1}^d   \int_{\R} \left[f(x + a_i(x) w) + f(x - a_i(x) w) - 2 f(x)\right] \nu_i(w) \, dw,
$$
for any Borel function $f: \R^d \to \R$ and any $x \in \R^d$ such that all the integrals on the right hand side are well defined. Recall that $a_i(x) = (a_{1i}(x),\ldots, a_{di}(x))$. It is well known that $\calK f(x)$ is well defined for any $f \in C_b^2(\Rd)$ and any $x \in \R^d$. By standard arguments, if $f \in C_c^2(\Rd)$, then $f(X_t) - f(X_0) - \int_0^t \calK f(X_s) \, ds$ is a martingale (see e.g. \cite[page 120]{K2011}). 

Let us fix $\eps \in (0,1]$ (it will be chosen later). 
For the given $\eps$ we choose the constant $\delta$ according to Lemma \ref{gtht1}. For such fixed $\eps$, $\delta$ we abbreviate $\mu_i(x) = \mu_i^{(\delta)}(x)$, $\calG_i = \calG_i^{(\delta)}$, $g_{i,t}(x) = g_{i,t}^{(\delta)}(x)$, 
$\tilde{g}_{i,t}(x) = \tilde{g}_{i,t}^{(\eps)}(x)$.

We divide $\calK$ into two parts
\begin{equation}
\label{KLR}
\calK f(x) = \calL f(x) + \calR f(x),
\end{equation}
where
$$
\calL f(x) =  \frac{1}{2} \sum_{i = 1}^d   \int_{\R} \left[f(x + a_i(x) w) + f(x - a_i(x) w) - 2 f(x)\right] \mu_i(w) \, dw.
$$ 
Our first aim will be to construct the heat kernel $u(t,x,y)$ corresponding to the operator $\calL$. This will be done by using Levi's method. 

For each $z \in \R^d$ we introduce the ``freezing'' operator 
$$
\calL^z f(x) =  \frac{1}{2} \sum_{i = 1}^d   \int_{\R} \left[f(x + a_i(z) w) + f(x - a_i(z) w) - 2 f(x)\right] \mu_i(w) \, dw,
$$ 

Let $G_t(x) = g_{1,t}(x_1) \ldots g_{d,t}(x_d)$ and $\tilde{G}_t(x) = \tilde{g}_{1,t}(x_1) \ldots \tilde{g}_{d,t}(x_d)$ for $t > 0$ and $x = (x_1,\ldots,x_d) \in \R^d$. We also denote $B(x) = (b_{ij}(x)) = A^{-1}(x)$. Note that the coordinates of $B(x)$ satisfy conditions 
(\ref{bounded}) and  (\ref{Lipschitz}) with possibly different constants $\eta_1^*$ and $\eta_3^*$, but taking maximums we can assume that $\eta_1^*=\eta_1$ and $\eta_3^*=\eta_3$.

For any $y \in \R^d$, $i = 1, \ldots, d$ we put
\begin{equation*}
b_i(y) = (b_{i1}(y),\ldots, b_{id}(y)).
\end{equation*}
We also denote $\|B\|_{\infty} = \max\{|b_{ij}|: \, i, j \in \{1,\ldots,d\}\}$.

For any $t > 0$, $x, y \in \R^d$ we define
\begin{eqnarray*}
p_y(t,x) &=& \det(B(y)) G_t(x(B(y))^T)\\
&=& \det(B(y)) g_{1,t}(b_1(y) x) \ldots g_{d,t}(b_d(y) x).
\end{eqnarray*}
It may be easily checked that for each fixed $y \in \R^d$ the function $p_y(t,x)$ is the heat kernel of $\calL^y$ that is 
$$
\frac{\partial}{\partial t} p_y(t,x) = \calL^y p_y(t,\cdot)(x), \quad t > 0, \, x \in \R^d,
$$
$$
\int_{\R^d} p_y(t,x) \, dx = 1, \quad t > 0.
$$
For any $t > 0$, $x, y \in \R^d$ we also define
\begin{eqnarray*}
r_y(t,x) &=&  \tilde{G}_t(x(B(y))^T)\\
&=&  \tilde{g}_{1,t}(b_1(y) x) \ldots \tilde{g}_{d,t}(b_d(y) x).
\end{eqnarray*}
For $x, y \in \R^d$, $t > 0$, let
$$
q_0(t,x,y) = \calL^{x}p_y(t,\cdot)(x-y) - \calL^{y} p_y(t,\cdot)(x-y),
$$
and for $n \in \N$ let
\begin{equation}
\label{defqn}
q_n(t,x,y) = \int_0^t \int_{\R^d} q_0(t-s,x,z)q_{n-1}(s,z,y) \, dz \, ds.
\end{equation}
For $x, y \in \R^d$, $t > 0$ we define
\begin{equation}
\nonumber
q(t,x,y) = \sum_{n = 0}^{\infty} q_n(t,x,y)
\end{equation}
and 
\begin{equation}
\label{defu}
u(t,x,y) = p_y(t,x-y) + \int_0^t \int_{\R^d} p_z(t-s,x-z)q(s,z,y) \, dz \, ds.
\end{equation}

In this section we will show that $q_n(t,x,y)$, $q(t,x,y)$, $u(t,x,y)$ are well defined and we will obtain estimates of these functions. First, we will get some simple properties of $p_y(t,x)$ and $r_y(t,x)$.

\begin{lemma}
\label{pyholder}
 For any $t \in (0,\tau]$, $x, x', y \in \R^d$ we have
$$
|p_y(t,x) - p_y(t,x')| \le c \left[1\wedge \left(\sum_{j = 1}^d \frac{|x_j - x'_j|}{h_j^{-1}(1/t)}\right) \, \right] 
\left(r_y(t,x/2) + r_y(t,x'/2)\right).
$$
\end{lemma}
The proof is very similar to the proof of \cite[Lemma 3.1]{KRS2018} and it is omitted.

\begin{lemma}
\label{estimate_pytx} Assume that $ \eps \le \frac1{ \eta_1 d\sqrt{d}}$. 
For any $t \in (0,\tau+1]$, $x, y \in \R^d$, we have
\begin{equation}
\label{htht1}
r_y(t,x-y)
\le c_1 \left(\prod_{i = 1}^d \frac{1}{h_i^{-1}(1/t)}\right) e^{-c |x - y|}.
\end{equation}
For any $t \in (0,\tau+1]$, $x, y \in \R^d$, $|x-y| \ge \eps\eta_1 d^{3/2}$, we have
\begin{equation}
\label{htht2}
r_y(t,x-y)
\le c_1 t e^{-c |x - y|}.
\end{equation}
\end{lemma}
The proof is almost the same as the proof of \cite[Corollary 3.3.]{KRS2018}, so we do not repeat it.

Using the definition of $p_y(t,x)$ and properties of $g_t(x)$ we obtain the following regularity properties of $p_y(t,x)$.
\begin{lemma}
\label{pycontinuity}
The function $(t,x,y) \to p_y(t,x)$ is continuous on $(0,\infty) \times \R^d \times \R^d$. The function $t \to p_y(t,x)$ is in $C^1((0,\infty))$ for each fixed $x, y \in \R^d$. The function $x \to p_y(t,x)$ is in $C^2(\R^d)$ for each fixed $t > 0$, $y \in \R^d$.
\end{lemma}

\begin{lemma}
\label{properties_pytx}
For any $y \in \R^d$ we have
$$
\left|\frac{\partial}{\partial x_i} p_y(t,x-y) \right| \le \frac{c}{t^{(d+1)/\alpha} (1+|x -y|)^{d+1}}, 
\quad i \in \{1,\ldots,d\}, \, t \in (0,\tau], \, x \in \R^d,
$$ 
$$
\left|\frac{\partial^2}{\partial x_i \partial x_j} p_y(t,x-y) \right| \le \frac{c}{t^{(d+2)/\alpha} (1+|x - y|)^{d+1}}, 
\quad i, j \in \{1,\ldots,d\}, \, t \in (0,\tau], \, x \in \R^d.
$$
\end{lemma}
\begin{proof}
The estimates follow from properties of $g_t(x)$ and Lemma \ref{gtht1} and the same arguments as in the proof of \cite[Corollary 3.3.]{KRS2018}.
\end{proof}

Let  $f:\R^n \to \R^n, n\in \N$, be a Lipschitz function. 
It is well known that $y$ almost surely the Jacobi matrix ${J}_{f}(y)$ of $f$ exists. For any  $y_0\in \R^n$ we define (see Definition 1 in \cite{C1976}) the generalized Jacobian denoted $\partial f(y_0)$ as the convex hull of the set of matrices which can be obtained as limits of ${J}_{f}(y_n)$, when $y_n\to y_0$.  

Now, we recall two results from \cite{KRS2018} which will be useful in the sequel.
\begin{lemma}\label{intA} \cite[Lemma 3.6]{KRS2018}

Let  $ b^*_{i}(x,y), x, y \in \Rd; i=1, \dots, d$, be real functions such that there are positive  $\eta_5, \eta_6 \ge 1$ and 
\begin{equation}
\label{bounded1}
|b^*_{i}(x,y)| \le \eta_5, \quad x, y \in \R^d,
\end{equation}
\begin{equation}
\label{Lipschitz1}
|b^*_{i}(x,y) - b^*_{i}(\overline{x},\overline{y})| \le \eta_6(|x-\overline{x}|+ |y-\overline{y}|), \quad x, y, \overline{x}, \overline{y} \in \R^d.
\end{equation}
Let, for fixed $x\in \R^d$, $\Psi_x$ be a map $\R^{d+1}\mapsto  \R^{d+1}$ given by

$$\Psi_x(w, y)= (w, \xi_1, \dots, \xi_d)\in \R^{d+1}, \quad w \in \R, y\in \R^{d}, $$
where  $\xi_i=b_i(y)(x-y) +b^*_{i}(x,y)w$.

There is a positive  $\eps_0= \eps_0(\eta_1,\eta_3,\eta_5,\eta_6,d)\le \frac1{2\eta_6}$ such that 
the map $\Psi_x$  and its   Jacobian determinant   denoted by $J_{\Psi_x}(w,y)$ has   the property
\begin{eqnarray*}
|\Psi_x(w, y)|&\le& 1,\\
(1/2) |\det B(y)| \le |J_{\Psi_x}(w,y)|&\le& 2|\det B(y)|, \end{eqnarray*}
 for $ |x-y|\le \eps_0, |w|\le \eps_0$, $(w,y)$ almost surely. Moreover the map $\Psi_x$ is injective on the set $\{(w,y)\in \R^{d+1}; |x-y|\le \eps_0, |w|\le \eps_0\}$.

If, for fixed $y\in \R^d$, $\Phi_y$ be a map $\R^{d+1}\mapsto  \R^{d+1}$ given by 
$$\Phi_y(w, x)=\Psi_x(w, y), \quad w \in \R, x\in \R^{d}, $$ then 
the  Jacobian of $\Phi_y$ denoted by $J_{\Phi_y}(w,x)$ has the property
$$(1/2) |\det B(y)| \le |J_{\Phi_y}(w,x)|\le 2|\det B(y)|,$$
for $ |x-y|\le \eps_0, |w|\le \eps_0$, $(w,x)$ almost surely.  Moreover the map $\Phi_y$ is injective on the set $\{(w,x)\in \R^{d+1}; |x-y|\le \eps_0, |w|\le \eps_0\}$.
\end{lemma}

\begin{remark} \cite[Remark 3.7]{KRS2018}

\label{intA1}Let for $x\in \R^d$, $\tilde{\Psi}_x$ be the map $\R^{d}\mapsto  \R^{d}$ given by 
$$\tilde{\Psi}_x( y)= (\xi_1, \dots, \xi_d)\in \R^{d}, \quad  y\in \R^{d}, $$
where  $\xi_i=b_i(y)(x-y)$. Then we can find $\eps_0$ such  that all the assertions of Lemma \ref{intA} are true and additionally
$$ (1/2) |\det B(y)| \le |J_{\tilde{\Psi}_x}(y)|\le 2|\det B(y)|,$$
for $ |x-y|\le \eps_0$, $y$ almost surely. Moreover,  the map $\tilde{\Psi}_x$ is injective on $B(x, \eps_0)$. We can also find $\delta_1= \delta_1(\eta_1, \eta_3, \eta_5, \eta_6, d)>0$ and  $\delta_2= \delta_1(\eta_1, \eta_3, \eta_5, \eta_6, d)>0$ such that the $\tilde{\Psi}_x$ image of the ball $B(x, \delta_1)$  contains $B(0, \delta_2)$. 
\end{remark}

Let $b^*_i(x,y)$ be the functions introduced in Lemma \ref{intA}.
We will use the following abbreviations
\begin{eqnarray*}
&& z_i=B_i(x,y) = b_i(y)(x-y) = b_{i1}(y)  (x_1-y_1) + ... + b_{id}(y) (x_d-y_d),\\
&&b^*_i= b^*_i(x,y),\\
&&b^*_{i0}= b^*_i(x,x). 
\end{eqnarray*}
Let for $k,l,m\in \{1,...,d\}$, 
\begin{equation*}
\text{A}_{l,m} = {A}_{l,m}(x,y)= \int_{\R}  \prod_{i\ne l} g_{i,t}(z_i +b^*_{i}w) \left|g_{l,t}(z_{l} \pm b^*_{l}w) - g_{l,t}(z_{l} \pm b^*_{l0}w)\right|
  \mu_m(w) \, dw.
\end{equation*}
For $l\ne k$ we denote 
\begin{eqnarray*}
\text{B}_{l,k,m} &=&\text{B}_{l,k,m}(x,y)\\&=& \int_{\R}  \prod_{i\ne l,k} g_{i,t}(z_i +b^*_{i}w) \left|g_{l,t}(z_{l} + b^*_{l}w) - g_{l,t}(z_{l} + b^*_{l0}w)\right|\\
&& \times \left|g_{k,t}(z_{k} + b^*_{k}w) - g_{k,t}(z_{k} - b^*_{k}w)\right| \mu_m(w) \, dw.
\end{eqnarray*}
 
When the assumptions (Z1) are satisfied we put $\sigma = 1 - \alpha/(3\beta)$, wile under the assumptions (Z2)  we put $\sigma = 2\beta /(3\alpha)$. Clearly, in both cases $\sigma \in (0,1)$.

\begin{corollary} \label{int10} Assume that $2\delta< \eps_0$, where $\eps_0$ is from Lemma \ref{intA}. With the assumptions of Lemma \ref{intA} we have for $t\le \tau$, $k, l, m \in \{1,\ldots,d\}$, $k \ne l$
$$\int_{|y-x|\le \eps_0}\left[\text{A}_{l,m}+\text{B}_{l,k,m}\right]dy\le ct^{-\sigma}, \quad x \in \R^d,$$
and 
$$\int_{|y-x|\le \eps_0}\left[\text{A}_{l,m}+\text{B}_{l,k,m}\right]dx\le ct^{-\sigma}, \quad y \in \R^d,$$
where $c = c(\tau, \alpha, d, \eta_1, \eta_2, \eta_3, \eta_4, \eta_5, \eta_6, \eps, \delta, \nu_0)$.
\end{corollary}
\begin{proof} In the proof we assume that constants $c$ may additionally depend on $\eta_5, \eta_6$.  It is enough to prove the estimates for $l=1$ and $k=2$.
For $x,y\in\Rd$ we get  $|b^*_{1}-b^*_{10}|\le \eta_6 |x-y|$. Hence,  from (\ref{gdelta4}), we have for $w\in \R$, 
 \begin {eqnarray*} && \left|g_{1,t}(z_{1} \pm b^*_{1}w) - g_{1,t}(z_{1} \pm b^*_{10}w)\right| \\
 &\le& c \left(\frac {|b^*_{1}-b^*_{10}|w^2}{\left(h_1^{-1}(1/t)\right)^2}\wedge 1\right)(g^*_{1,t}(z_{1} \pm b^*_{1}w) + g^*_{1,t}(z_{1} \pm b^*_{10}w)) \\
&\le& c \left(\frac {|x-y|w^2}{\left(h_1^{-1}(1/t)\right)^2}\wedge 1\right)(g^*_{1,t}(z_{1} \pm b^*_{1}w) + g^*_{1,t}(z_{1} \pm b^*_{10}w)).  
\end{eqnarray*}

This implies that 
   $$\text{A}_{1,m}\le c(\text{A}^1_{1,m}+\text{A}^2_{1,m}+\text{A}^3_{1,m}+\text{A}^4_{1,m}),$$
where 
$$
\text{A}^r_{1,m}=   \int_{\R}  \prod_{i=1}^{d}g^{*}_{i,t}(z_i + \hat{b}^{r}_{i}w)  \left(\frac {|x-y|w^2}{\left(h_1^{-1}(1/t)\right)^2}\wedge 1\right)\mu_m(w) \, dw $$ 
with $\hat{b}^{r}_{i}= b^*_{i}, i\ge 2$ and  $\hat{b}^1_1= b^*_{1}$, $\hat{b}^2_1= -b^*_{1}$,  $\hat{b}^3_1= b^*_{10}$ and 
$\hat{b}^4_1= -b^*_{10}.$
Note that the functions $\hat{b}^{r}_{i}= \hat{b}^{r}_{i}(x,y) $ have the same properties (\ref{bounded1},  \ref{Lipschitz1}) as $b^*_{i}$.
To evaluate the integral $\int_{|x-y|\le \eps_0}\text{A}^1_{1,m}dy$ we introduce new variables in $\R^{d+1}$, given by  $(w,\xi)=\Psi_x(w, y)$, where $   \xi_i=z_i + b^*_iw, i=1,\dots,d$ (or $\xi_i=z_i + \hat{b}^r_iw$ if $\text{A}^r_{1,m}$ is treated for $r=2,3,4$). Note that the vector $\xi= (\xi_1, \dots, \xi_d) $ can be written as
$$\xi= (x-y)B(y)^T+ w b^*,$$
where $b^* = (b^*_1, \dots, b^*_d)$, hence
$$(\xi-w b^*)A(y)^T= x-y.$$ From this we infer that
$$|w|^2|x-y|\le c(|\xi|+|w|)|w|^2.$$

Let $Q_x=\{(w,y):|y-x|\le \eps_0, \ |w|\le \eps_0\}$.  Due to Lemma \ref{intA}, almost surely on $Q_x$, the absolute value of the  Jacobian determinant  of the map $\Psi_x$ is bounded from below and above by two positive constants and $\Psi_x$ is an injective transformation. Let $V_x= \Psi_x(Q_x)$. Observing that the support of the measure $\mu$ is contained in $[-\eps_0, \eps_0]$ and then  applying the above change of variables, we have 
\begin{eqnarray*}\int_{|y-x|\le \eps_0}\text{A}^1_{1,m}dy&\le& 
c \int_{|y-x|\le \eps_0} \int_{\R} \prod_{i=1}^{d} g^{*}_{i,t}(\xi_i )
 \left(\frac{(|\xi|+|w|)|w|^2}{ \left(h_1^{-1}(1/t)\right)^2}\wedge1\right)\mu_m(w) \, dw \, d y \\ 
&\le& c \int_{|y-x|\le \eps_0} \int_{\R} \prod_{i=1}^{d} g^{*}_{i,t}(\xi_i )
 \left(\frac{(|\xi|+|w|)|w|^2}{ \left(h_1^{-1}(1/t)\right)^2}\wedge1\right)\\
&& \times \,\, \mu_m(w)|J_{\Psi_x}(w,y)| \, dw \, d y \\ 
&=& c \int_{V_x}  \prod_{i=1}^{d} g^{*}_{i,t}(\xi_i )
 \left(\frac{(|\xi|+|w|)|w|^2}{ \left(h_1^{-1}(1/t)\right)^2}\wedge1\right)\mu_m(w) \, dw \, d \xi,\end{eqnarray*}
where the last equality follows from the general  change of variable formula  
for injective  Lipschitz maps (see e.g. \cite[Theorem 3]{H1993}). 
Since $|\xi|\le 1$ for $(w,\xi)\in V_x $,   we get 
$$\int_{|y-x|\le \eps_0}\text{A}^1_{1,m}dy\le c \int_{|\xi|\le1} \prod_{i=1}^{d} g^{*}_{i,t}(\xi_i )
\int_{\R}\left(\frac{(|\xi|+|w|)|w|^2}{ \left(h_1^{-1}(1/t)\right)^2}\wedge1\right)\mu_m(w) \, dw \, d \xi.   $$
Applying  (\ref{levy1b}) for $m=1$ we have for   $|\xi|\le 1$,
$$\int_{\R} \left(\frac{(|\xi|+|w|)|w|^2}{ \left(h_1^{-1}(1/t)\right)^2}\wedge1\right)\mu_1(w) \, dw \le c 
\frac{t^{\alpha/(3\beta)}+|\xi|^{\alpha/2}}{t}.
$$
Consequently, by Lemma \ref{momenty}, we obtain 
\begin{equation}
\label{A111}
\int_{|y-x|\le \eps_0}\text{A}^1_{1,1} \, dy\le c  \int_{|\xi|\le1} \prod_{i=1}^{d} g^{*}_{i,t}(\xi_i )\frac{t^{\alpha/(3\beta)}+|\xi|^{\alpha/2}}{t} \, d \xi  \le c t^{-(1-\alpha/(3\beta))}. 
\end{equation}
If the assumptions (Z1) are satisfied, then $\mu_1 = \ldots = \mu_d$ and $\sigma = 1 - \alpha/(3\beta)$, so
\begin{equation}
\label{A1Z1}
\int_{|y-x|\le \eps_0}\text{A}^1_{1,m} \, dy
\le c t^{- \sigma} \quad \text{for} \quad m = 1, \ldots, d. 
\end{equation}
Now assume that the assumptions (Z2) are satisfied. Then applying (\ref{levy2b}) for $m \ge 2$, we have for $|\xi| \le 1$,
$$\int_{\R} \left(\frac{(|\xi|+|w|)|w|^2}{ \left(h_1^{-1}(1/t)\right)^2}\wedge1\right)\mu_m(w) \, dw \le c \left[
|\xi|^{\beta/2} t^{-\beta/\alpha}+   {t^{-{2\beta}/(3\alpha)}}\right].
$$
By Lemma \ref{momenty}, we obtain
\begin{equation}
\label{A11m}
\int_{|y-x|\le \eps_0}\text{A}^1_{1,m}dy  \le c {t^{-{2\beta}/(3\alpha)}} + c t^{1/2 - \beta/\alpha} \le c {t^{-{2\beta}/(3\alpha)}}
\quad \text{for} \quad m = 2, \ldots, d. 
\end{equation}
By elementary arguments $-1+\alpha/(3\beta) > -{2\beta}/(3\alpha)$, 
so for $t \in (0,\tau]$ we have 
$t^{-1+\alpha/(3\beta)} \le c t^{-{2\beta}/(3\alpha)}$. Hence, when the assumptions (Z2) are satisfied, using (\ref{A111}), (\ref{A11m}) and the fact that $\sigma = 2\beta/(3 \alpha)$, we have
\begin{equation}
\label{A1Z2}
\int_{|y-x|\le \eps_0}\text{A}^1_{1,m} \, dy
\le c t^{- \sigma} \quad \text{for} \quad m = 1, \ldots, d. 
\end{equation}

In a similar way as (\ref{A1Z1}), (\ref{A1Z2}) were obtained, for the both assumptions (Z1), (Z2), we get
\begin{equation*}
\int_{|y-x|\le \eps_0}\text{A}^r_{1,m} \, dy
\le c t^{- \sigma} \quad \text{for} \quad m = 1, \ldots, d, \,\, r =2, 3, 4.
\end{equation*}
This completes the proof of the bound (for the both assumptions (Z1), (Z2))
\begin{equation*}
\int_{|y-x|\le \eps_0}\text{A}_{1,m} \, dy
\le c t^{- \sigma} \quad \text{for} \quad m = 1, \ldots, d.
\end{equation*}

For $x,y\in\Rd$ we get  $|b^*_{1}-b^*_{10}|\le \eta_6 |x-y|$. Hence,  from (\ref{gdelta2}), we have for $w\in \R$,

 \begin {eqnarray*}&&  \left|g_{1,t}(z_{1} + b^*_{1}w) - g_{1,t}(z_{1} + b^*_{10}w)\right|
\left|g_{2,t}(z_{2} + b^*_{2}w) - g_{2,t}(z_{2} - b^*_{2}w)\right| \\
&\le& c \left(\frac {|b^*_{1}-b^*_{10}||w|}{h_1^{-1}(1/t)}\wedge 1\right)\left(\frac {|w|}{h_2^{-1}(1/t)}\wedge 1\right)\\
&&\times \left(g^*_{1,t}(z_{1} + b^*_{1}w) + g^*_{1,t}(z_{1} + b^*_{10}w)\right)g^*_{2,t}(z_{2} \pm b^*_{2}w)   \\
&\le& c \left(\frac {|y-x||w|}{h_1^{-1}(1/t)}\wedge 1\right)\left(\frac {|w|}{h_2^{-1}(1/t)}\wedge 1\right)\\
&&\times \left(g^*_{1,t}(z_{1} + b^*_{1}w) + g^*_{1,t}(z_{1} + b^*_{10}w)\right)g^*_{2,t}(z_{2} \pm b^*_{2}w).  \\ 
\end{eqnarray*}
This implies that 
   $$\text{B}_{1,2,m}\le c(\text{B}^1_{1,2,m}+\text{B}^2_{1,2,m}+\text{B}^3_{1,2,m}+\text{B}^4_{1,2,m}),$$
where 
$$
\text{B}^r_{1,2,m}=   \int_{\R} \left( \prod_{i=1}^{d}g^{*}_{i,t}(z_i + \hat{b}^{r}_{i}w)\right) \left(\frac {|y-x||w|}{h_1^{-1}(1/t)}\wedge 1\right)\left(\frac {|w|}{h_2^{-1}(1/t)}\wedge 1\right)\mu_m(w) \, dw $$ 
with $\hat{b}^{r}_{i}= b^*_{i}, i\ge 3$ and  $\hat{b}^1_1=\hat{b}^2_1= b^*_{1}$, $\hat{b}^3_1=\hat{b}^4_1= b^*_{10}$ and
$\hat{b}^1_2= \hat{b}^3_2=-\hat{b}^2_2=-\hat{b}^4_2= b^*_{2}$. 
Note that the functions $\hat{b}^{r}_{i}= \hat{b}^{r}_{i}(x,y) $ have the same properties (\ref{bounded1},  \ref{Lipschitz1}) as $b^*_{i}$.

We proceed as before  and  introduce new variables in $\R^{d+1}$, given by  $(w,\xi)=\Psi_x(w, y)$, where $   \xi_i=z_i + \hat{b}^{r}_{i}w, i=1,\dots,d$. 
Again we have that
$$|w||x-y|\le c(|\xi|+|w|)|w|.$$

By the same arguments as before  
\begin{equation}
\label{B12m1}
\int_{|y-x|\le \eps_0}\text{B}^1_{1,2,m}dy
\le c \int_{|\xi|\le1} \prod_{i=1}^{d} g^{*}_{i,t}(\xi_i )\int_{\R}
 \left(\frac {(|\xi|+|w|)|w|^2}{h_1^{-1}(1/t) h_2^{-1}(1/t)}\wedge 1\right) \mu_m(w) \, dw \, d \xi.   
\end{equation}
If assumptions (Z1) are satisfied then $h_1 = h_2$, $\mu_1 = \ldots = \mu_d$ and $\sigma = 1 - \alpha/(3\beta)$. Repeating the arguments which give (\ref{A1Z1}) we get
\begin{equation}
\label{B12mZ1}
\int_{|y-x|\le \eps_0}\text{B}^1_{1,2,m} \, dy
\le c t^{- \sigma} \quad \text{for} \quad m = 1, \ldots, d. 
\end{equation}
If the assumptions (Z2) are satisfied, then by (\ref{B12m1}), (\ref{levy2a}) and Lemma \ref{momenty}, we get
\begin{equation}
\label{B12mZ2}
\int_{|y-x|\le \eps_0}\text{B}^1_{1,2,m} \, dy
\le c t^{1 - \beta/\alpha} \log(1+1/t) + c t^{-\beta/(2\alpha)} \le c t^{-2\beta/(3\alpha)} = c t^{-\sigma} 
\end{equation}
for  $m = 1, \ldots, d$. 
In a similar way as (\ref{B12mZ1}), (\ref{B12mZ2}) were obtained, for the both assumptions (Z1), (Z2), we get
\begin{equation*}
\int_{|y-x|\le \eps_0}\text{B}^r_{1,2,m} \, dy
\le c t^{- \sigma} \quad \text{for} \quad m = 1, \ldots, d, \,\, r =2, 3, 4.
\end{equation*}
This completes the proof of the bound (for the both assumptions (Z1), (Z2))
\begin{equation*}
\int_{|y-x|\le \eps_0}\text{B}_{1,2,m} \, dy
\le c t^{- \sigma} \quad \text{for} \quad m = 1, \ldots, d,
\end{equation*}
which finishes the proof of the first estimate.  

To estimate the second integral (with respect to $dx$) we proceed exactly in the same way.
\end{proof}

For fixed $l \in \{1,\ldots,d\}$ let us consider a family  of functions $b_i^*(x,y)= b_{i}(y) a_{l}(x), i \in \{1,\dots,d\}$. They satisfy the conditions (\ref{bounded1}) and (\ref{Lipschitz1}) with $\eta_5= d\eta_1^2$ and 
$\eta_6= d\eta_1\eta_3$. Let  $\eps_0=\eps_0(\eta_1,\eta_3,\eta_5,\eta_6,d)$ be as found in Lemma \ref{intA} and Remark \ref{intA1}. 
Finally we choose $\eps=\eps(\eta_1,\eta_3,d)= \frac{\eps_0}{4d^{3/2} \eta_1} $. From now on we keep 
$\eps_0, \eps$ fixed as above. Recall that if we fixed $\eps$ we fix $\delta$ according to Lemma \ref{gtht1}.

\begin{lemma}\label{difference}
For any $i \in \{1, \ldots, d\}$ and $\mathcal{a}_i, \mathcal{b}_i, \mathcal{c}_i, \mathcal{d}_i \in \R$ we have
\begin{eqnarray}
\nonumber
&& \prod_{i=1}^{d} \mathcal{a}_i + \prod_{i=1}^{d} \mathcal{b}_i  - \prod_{i=1}^{d} \mathcal{c}_i - \prod_{i=1}^{d} \mathcal{d}_i\\
\nonumber
&=&  \sum_{j=1}^{d} \left[
\sum_{k=1}^{j-1} \left(\prod_{i=1}^{k-1} \mathcal{d}_i\right) \left(\mathcal{c}_k - \mathcal{d}_k\right) \left(\prod_{i=k+1}^{j-1} \mathcal{c}_i\right)
\left(\mathcal{a}_j-\mathcal{c}_j\right) \left(\prod_{i=j+1}^{d} \mathcal{a}_i\right) \right. \\
\nonumber
&+&   \left(\prod_{i=1}^{j-1} \mathcal{d}_i\right) \left(\mathcal{a}_j-\mathcal{c}_j - (\mathcal{d}_j - \mathcal{b}_j)\right) \left(\prod_{i=j+1}^{d} \mathcal{a}_i\right)\\
\label{difference_formula}
&+& \left.  
\sum_{k=j+1}^{d} \left(\prod_{i=1}^{j-1} \mathcal{d}_i\right) \left(\mathcal{d}_j-\mathcal{b}_j\right) \left(\prod_{i=j+1}^{k-1} \mathcal{b}_i\right)
\left(\mathcal{a}_k-\mathcal{b}_k\right) \left(\prod_{i=k+1}^{d} \mathcal{a}_i\right) \right].
\end{eqnarray}
We understand here that for $m > n$ we have $\prod_{i=m}^{n} \mathcal{e}_i = 1$ and $\sum_{i=m}^{n} \mathcal{e}_i = 0$.
\end{lemma}
\begin{proof}
We observe that
\begin{equation}
\label{difference1}
\prod_{i=1}^{d} \mathcal{a}_i - \prod_{i=1}^{d} \mathcal{c}_i =
\sum_{j=1}^{d} \left(\prod_{i=1}^{j-1} \mathcal{c}_i\right) \left(\mathcal{a}_j-\mathcal{c}_j\right) \left(\prod_{i=j+1}^{d} \mathcal{a}_i\right).
\end{equation}
Similarly, we obtain
\begin{equation*}
\prod_{i=1}^{d} \mathcal{b}_i - \prod_{i=1}^{d} \mathcal{d}_i =
\sum_{j=1}^{d} \left(\prod_{i=1}^{j-1} \mathcal{d}_i\right) \left(\mathcal{b}_j-\mathcal{d}_j\right) \left(\prod_{i=j+1}^{d} \mathcal{b}_i\right),
\end{equation*}
so
\begin{equation}
\label{difference2}
\prod_{i=1}^{d} \mathcal{d}_i - \prod_{i=1}^{d} \mathcal{b}_i =
\sum_{j=1}^{d} \left(\prod_{i=1}^{j-1} \mathcal{d}_i\right) \left(\mathcal{d}_j-\mathcal{b}_j\right) \left(\prod_{i=j+1}^{d} \mathcal{b}_i\right).
\end{equation}
By (\ref{difference1}) and (\ref{difference2}) we get
\begin{eqnarray}
\nonumber
&& \prod_{i=1}^{d} \mathcal{a}_i + \prod_{i=1}^{d} \mathcal{b}_i  - \prod_{i=1}^{d} \mathcal{c}_i - \prod_{i=1}^{d} \mathcal{d}_i\\
\nonumber
&=& \left(\prod_{i=1}^{d} \mathcal{a}_i - \prod_{i=1}^{d} \mathcal{c}_i \right)  - \left( \prod_{i=1}^{d} \mathcal{d}_i - \prod_{i=1}^{d} \mathcal{b}_i \right)\\
\label{difference3}
&=& \sum_{j=1}^{d} \left(\prod_{i=1}^{j-1} \mathcal{c}_i\right) \left(\mathcal{a}_j-\mathcal{c}_j\right) \left(\prod_{i=j+1}^{d} \mathcal{a}_i\right)
- \sum_{j=1}^{d} \left(\prod_{i=1}^{j-1} \mathcal{d}_i\right) \left(\mathcal{d}_j-\mathcal{b}_j\right) \left(\prod_{i=j+1}^{d} \mathcal{b}_i\right).
\end{eqnarray}
For any $j \in \{1,\ldots,d\}$ we have
\begin{eqnarray}
\nonumber
&& \left(\prod_{i=1}^{j-1} \mathcal{c}_i\right) \left(\mathcal{a}_j-\mathcal{c}_j\right) \left(\prod_{i=j+1}^{d} \mathcal{a}_i\right)
- \left(\prod_{i=1}^{j-1} \mathcal{d}_i\right) \left(\mathcal{d}_j-\mathcal{b}_j\right) \left(\prod_{i=j+1}^{d} \mathcal{b}_i\right)\\
\nonumber
&=&  
\sum_{k=1}^{j-1} \left(\prod_{i=1}^{k-1} \mathcal{d}_i\right) \left(\mathcal{c}_k - \mathcal{d}_k\right) \left(\prod_{i=k+1}^{j-1} \mathcal{c}_i\right)
\left(\mathcal{a}_j-\mathcal{c}_j\right) \left(\prod_{i=j+1}^{d} \mathcal{a}_i\right)\\
\nonumber
&+&   \left(\prod_{i=1}^{j-1} \mathcal{d}_i\right) \left(\mathcal{a}_j-\mathcal{c}_j - (\mathcal{d}_j - \mathcal{b}_j)\right) \left(\prod_{i=j+1}^{d} \mathcal{a}_i\right)\\
\label{difference4}
&+&  
\sum_{k=j+1}^{d} \left(\prod_{i=1}^{j-1} \mathcal{d}_i\right) \left(\mathcal{d}_j-\mathcal{b}_j\right) \left(\prod_{i=j+1}^{k-1} \mathcal{b}_i\right)
\left(\mathcal{a}_k-\mathcal{b}_k\right) \left(\prod_{i=k+1}^{d} \mathcal{a}_i\right).
\end{eqnarray}
Now, (\ref{difference3}) and (\ref{difference4}) give (\ref{difference_formula}).
\end{proof}

\begin{proposition}
\label{integralq02}
  For any $x,y \in \R^d$, $t \in (0,\tau]$  we have 
\begin{equation}\label{global}
|q_0(t,x,y)| \le    c\frac 1{t^{\beta/\alpha+d/\alpha}}.
\end{equation}
 For  $x,y \in \R^d$, $t \in (0,\tau]$, $|y-x|\ge \eps_0$ we have\begin{equation}\label{exp}
|q_0(t,x,y)| \le    c e^{-(\eps/\eps_0)|x-y|}.
\end{equation}
For any $t \in (0,\tau]$, $x \in \R^d$ we have
\begin{equation}\label{q0int}
\int_{\R^d} |q_0(t,x,y)| \, dy \le c t^{-\sigma}.
\end{equation}
For any $t \in (0,\tau]$, $y \in \R^d$ we have

\begin{equation}\label{q0intx}
\int_{\R^d} |q_0(t,x,y)| \, dx \le c t^{-\sigma}.
\end{equation}

\end{proposition}
\begin{proof}
We have
\begin{eqnarray}
\nonumber
 q_0(t,x,y) &=&  \frac{1}{2} \sum_{i = 1}^d   \int_{\R} \left[p_y(t,x-y + a_i(x) w) 
+ p_y(t,x-y - a_i(x) w)\right. \\
\label{q0formula}
&& \left. -p_y(t,x-y + a_i(y) w) 
- p_y(t,x-y - a_i(y) w)\right] \, \mu_i(w) \, dw.
\end{eqnarray}
For $i = 1, \ldots, d$ we put 
\begin{eqnarray} 
\nonumber
   R_i &=&  \frac{1}{2} \int_{\R} \left[p_y(t,x-y + a_i(x) w) 
+ p_y(t,x-y - a_i(x) w)\right. \\
\label{term1}
&& \left. -p_y(t,x-y + a_i(y) w) 
- p_y(t,x-y - a_i(y) w)\right] \, \mu_i(w) \, dw.
\end{eqnarray}
We have $q_0(t,x,y) =  R_1 + \ldots + R_d$. It is clear that it is enough to handle $R_1$ alone.
Note that 
\begin{eqnarray}
\nonumber
  R_1  &=& \frac{1}{2} \det(B(y)) \int_{\R} \left[G_t\left((x-y + w e_1 (A(x))^T) (B(y))^T\right) \right.\\
\nonumber 
&&+ G_t\left((x-y - w e_1 (A(x))^T) (B(y))^T\right)\\
\nonumber 
&&  - G_t\left((x-y + w e_1 (A(y))^T) (B(y))^T\right)\\
\label{formulaR1}  
&& \left. - G_t\left((x-y - w e_1 (A(y))^T) (B(y))^T\right)
\right] \, \mu_1(w) \, dw.
\end{eqnarray}

We will use the following abbreviations
\begin{eqnarray*}
&& z_i=B_i(x,y) = b_i(y)(x-y) = b_{i1}(y)  (x_1-y_1) + ... + b_{id}(y) (x_d-y_d),\\
&&k_i= \tbfs{i} = b_{i}(y) a_{1}(x),\\ 
&&k_{i0}=\tilde{b}_{i1}(x,x). \end{eqnarray*}
Note that $k_{10}=1$ and $ k_{i0}=0,\ 2\le i\le d$.

Let $$\delta_t(w)=\prod_{i=1}^{d} g_{i,t}(z_i + k_i w)+
	\prod_{i=1}^{d} g_{i,t}(z_i - k_i w)-\prod_{i=1}^{d} g_{i,t}(z_i + k_{i0} w)-\prod_{i=1}^{d} g_{i,t}(z_i - k_{i0} w).$$
We can rewrite (\ref{formulaR1}) as
\begin{eqnarray*}
   R_1&=& \frac{1}{2} \det(B(y))   \int_{\R} \delta_t(w)\, \mu_1(w) \, dw.
\end{eqnarray*}
%
%
 By Lemma \ref{difference},  denoting $$
\mathcal{a}_i=g_{i,t}(z_i+k_i w), \,\,\,
\mathcal{b}_i=g_{i,t}(z_i-k_i w), \,\,\,
\mathcal{c}_i=g_{i,t}(z_i+k_{i0} w), \,\,\,
\mathcal{d}_i=g_{i,t}(z_i-k_{i0} w),
$$ we have

\begin{equation}
\label{difference_formula1}
 \delta_t(w) = \sum_{j=1}^{d} 
\left[ \left(\sum_{k=1}^{j-1} \delta^{k,j}_t(w)\right) + \delta^{j,j}_t(w) + \left(\sum_{k=j+1}^{d} \delta^{k,j}_t(w)\right)\right],
\end{equation}
where
\begin{eqnarray*}
\nonumber
\delta^{k,j}_t(w)&=&  
 \left(\prod_{i=1}^{k-1} \mathcal{d}_i\right) \left(\mathcal{c}_k - \mathcal{d}_k\right) \left(\prod_{i=k+1}^{j-1} \mathcal{c}_i\right)
\left(\mathcal{a}_j-\mathcal{c}_j\right) \left(\prod_{i=j+1}^{d} \mathcal{a}_i\right), k<j, \\
\nonumber
\delta^{j,j}_t(w)&=&   \left(\prod_{i=1}^{j-1} \mathcal{d}_i\right) \left(\mathcal{a}_j-\mathcal{c}_j - (\mathcal{d}_j - \mathcal{b}_j)\right) \left(\prod_{i=j+1}^{d} \mathcal{a}_i\right),\\
\delta^{k,j}_t(w) &=&  
 \left(\prod_{i=1}^{j-1} \mathcal{d}_i\right) \left(\mathcal{d}_j-\mathcal{b}_j\right) \left(\prod_{i=j+1}^{k-1} \mathcal{b}_i\right)
\left(\mathcal{a}_k-\mathcal{b}_k\right) \left(\prod_{i=k+1}^{d} \mathcal{a}_i\right), k>j .
\end{eqnarray*}

We denote $M_{i,t}=\max_{|w|\le 2\delta }
 \tilde{g}_{i,t}(z_i + k^*w )$, where  $k^*= max \{1,|k_1|, |k_2|, \dots |k_d|\} $.
 By  (\ref{gdelta42}), for $k < j$ we have 
$$
\int_{\R}|\delta^{k,j}_t(w)|\mu_1(w)\, dw 
\le  \prod_{i\ne j,k } M_{i,t}\int_{\R}|\left(\mathcal{c}_k - \mathcal{d}_k\right) \left(\mathcal{a}_j-\mathcal{c}_j\right)|
\mu_1(w)\, dw 
\le c  t^{-\beta/\alpha} \prod_{i=1 }^{d} M_{i,t}.
$$
Similarly, for $k > j$ we have
$$
\int_{\R}|\delta^{k,j}_t(w)|\mu_1(w)\, dw 
\le  \prod_{i\ne j,k } M_{i,t}\int_{\R}|\left(\mathcal{d}_j - \mathcal{b}_j\right) \left(\mathcal{a}_k-\mathcal{b}_k\right)|
\mu_1(w)\, dw 
\le c  t^{-\beta/\alpha} \prod_{i=1 }^{d} M_{i,t}.
$$
By  (\ref{gdelta41}) we obtain
$$
\int_{\R}|\delta^{j,j}_t(w)|\mu_1(w)\, dw 
\le  \prod_{i\ne j } M_{i,t}\int_{\R} |\mathcal{a}_j-\mathcal{c}_j - (\mathcal{d}_j - \mathcal{b}_j)| 
\mu_1(w)\, dw 
\le c  t^{-\beta/\alpha} \prod_{i=1 }^{d} M_{i,t}.
$$
It follows that
$$
|R_1|\le c   t^{-\beta/\alpha} \prod_{i=1 }^{d} M_{i,t}.
$$
Since $M_{i,t}\le \frac c{t^{1/\alpha} }$ we obtain (\ref{global}) and moreover 
$$
|R_1|\le c  \min_i M_{i,t} \, t^{-\frac{d-1+\beta}\alpha}.
$$
By \cite[Lemma 3.2]{KRS2018}, $\max_i|z_i| \ge \frac1{\eta_1d^{3/2}}|x-y|$  and suppose that  
$|z_1| \ge \frac1{\eta_1d^{3/2}}|x-y|$. Then, since   $|k^*|\le \eta_1^2d$, we have  for $|x-y|\ge 4d^{5/2}\eta_1^3\delta $ and $|w|\le 2\delta$,
\begin{eqnarray*}
|z_1 + k^*w|&\ge&|z_1| - |k^*w|\ge  \frac1{\eta_1d^{3/2}}|x-y|- 2\eta_1^2 d \delta\\& =&\frac1{\eta_1d^{3/2}}|x-y|\left(1- 2d^{5/2}\frac{\eta_1^3 \delta}{|x-y|}
\right) \ge  \frac1{2\eta_1d^{3/2}}|x-y|.
\end{eqnarray*}
This yields that 
\begin{equation*}
|R_1| \le 
 c t^{-\beta/\alpha -(d-1)/\alpha}\tilde{g}_{1,t}\left(\frac{|x-y|}{2\eta_1d^{3/2}}\right),   \  |x-y|\ge 4d^{5/2}\eta_1^3\delta 
\end{equation*}
and provides the exponential bound 
\begin{equation*}
|R_1| \le  c e^{-\left(\frac{|x-y|}{2\eta_1d^{3/2}}\right)}, \quad {|x-y|}\ge \max\{2d^{3/2} \eta_1\eps, 4d^{5/2}\eta_1^3\delta\}.
\end{equation*}
Recall that $\eps= \frac{\eps_0}{4d^{3/2}\eta_1} $ and $\delta = \min\{\delta_0, \frac{\eps \alpha}{8(d+\beta + 2)},\frac{\eps}{d\eta_1^2}\}$. Hence\\ $\max\{2d^{3/2} \eta_1\eps, 4d^{5/2}\eta_1^3\delta\}\le \eps_0$, so finally
\begin{equation*}
|R_1| \le  c e^{-\frac{|x-y|}{2\eta_1d^{3/2}}}, \quad {|x-y|}\ge \eps_0,
\end{equation*}
which proves (\ref{exp}).

The estimates  (\ref{q0int}) and  (\ref{q0intx}) follow from Corollary \ref{int10} and (\ref{exp}). For example to handle the integral 
$$
\int_{|y-x|\le \eps_0}\int_R |\delta^{j,j}_t(w)| \mu_1(w) \, dw \, dy , \quad x \in \R^d,
$$ we take
$$
b_i^*(x,y)= -k_{i0}(x,x), i=1,\dots, j-1, \quad  b_i^*(x,y)= k_{i}(x,y), i=j,\dots, d.
$$
Such choice of functions $b_i^*$ enable us to apply Corollary \ref{int10}, since they satisfy all the assumptions of Lemma \ref{intA}. Hence
$$\int_{|y-x|\le \eps_0}\int_R |\delta^{j,j}_t(w)|\mu_1(w) \, dw \, dy \le c t^{-\sigma}.$$
The same argument (with an appropriate choice of $b_i^*$) shows that for $k \ne j$
$$\int_{|y-x|\le \eps_0}\int_R |\delta^{k,j}_t(w)|\mu_1(w) \, dw \, dy \le c t^{-\sigma}.$$
This implies that 
$$\int_{|y-x|\le \eps_0}|q_0(t,x,y)| \, dy \le c t^{-\sigma}.$$
By (\ref{exp}) we can extend the domain of integration to the whole $\R^d$ keeping the upper bound as above. 
\end{proof}
Using Corollary \ref{int_tilde_q} and similar arguments as in the proof of Proposition 3.10 in \cite{KRS2018} we obtain the following result.
\begin{proposition}
\label{pyintegral}
For any $t \in (0,\tau]$, $x \in \R^d$ we have
\begin{equation} \label{pyintegral1}
\int_{\R^d} p_y(t,x-y) \, dy \le c,
\end{equation}
\begin{equation} \label{ryintegral1}
\int_{\R^d} r_y(t,(x-y)/2) \, dy \le c.
\end{equation}
For any  $\delta_1 > 0$, 
\begin{equation}
\label{supsup}
\lim_{t\to 0^+} \sup_{x \in \R^d} \int_{B^c(x,\delta_1)} p_y(t,x-y) \, dy =0.
\end{equation}
Moreover,
\begin{equation}
\label{limp_y}
\lim_{t \to 0^+} \int_{\R^d} p_y(t,x-y) \, dy = 1,
\end{equation}
uniformly with respect to $x \in \R^d$.
\end{proposition}

In the sequel we will use the following standard estimate. For any $\gamma \in (0,1]$, $\theta_0 > 0$ there exists $c = c(\gamma,\theta_0)$ such that for any $\theta \ge \theta_0$, $t > 0$ we have
\begin{equation}
\label{gammatheta}
\int_0^t (t - s)^{\gamma-1} s^{\theta - 1} \, ds \le \frac{c}{\theta^{\gamma}} t^{(\gamma-1)+(\theta-1)+1}.
\end{equation}

\begin{lemma}
\label{integralqn}
For any $t > 0$, $x \in \R^d$ and $n \in \N$ the kernel $q_n(t,x,y)$ is well defined. For any $t \in (0,\tau]$, $x \in \R^d$ and $n \in \N$ we have
\begin{equation}
\label{integralqn1}
\int_{\R^d} |q_n(t,x,y)| \, dy \le \frac{c_1^{n+1} t^{(n+1)(1-\sigma) - 1}}{(n!)^{1 - \sigma}},
\end{equation}
\begin{equation}
\label{integralqn2}
\int_{\R^d} |q_n(t,y,x)| \, dy \le \frac{c_1^{n+1} t^{(n+1)(1-\sigma) - 1}}{(n!)^{(1-\sigma)}}.
\end{equation}
For any $t \in (0,\tau]$, $x, y \in \R^d$ and $n \in \N$ we have
\begin{equation}
\label{qnestimate1}
|q_n(t,x,y)|  \le c_1 \frac{c_2^n t^{n(1-\sigma) - 1}}{(n!)^{(1-\sigma)} t^{-1+(d+\beta)/\alpha}}.
\end{equation}
For any $t \in (0,\tau]$, $x, y \in \R^d$ and $n \in \N$, $|x - y| \ge n + 1$ we have
\begin{equation}
\label{qnestimate2}
|q_n(t,x,y)|  \le c_1 \frac{c_2^n t^{n(1-\sigma)}}{(n!)^{(1-\sigma)}} e^{-\frac{\lambda |x - y|}{n+1}},
\end{equation}
where $\lambda=\eps/\eps_0$.
\end{lemma}
\begin{proof} By  Proposition \ref{integralq02} there is a constant $c^* \ge 1$ such that 
for any $x,y \in \R^d$, $t \in (0,\tau]$  we have 
\begin{equation}\label{global1}
|q_0(t,x,y)| \le    c^*\frac 1{t^{(d+\beta)/\alpha}},
\end{equation}
\begin{equation}\label{global2}
|q_0(t,x,y)| \le    c^* e^{-\lambda|x-y|}, |x-y|\ge 1 .
\end{equation}
\begin{equation}\label{q0int1}
\int_{\R^d} |q_0(t,x,u)| \, du \le c^* t^{- \sigma},
\end{equation}
\begin{equation}\label{q0intx1}
\int_{\R^d} |q_0(t,u,x)| \, du \le c^* t^{ - \sigma}.
\end{equation}
It follows from (\ref{gammatheta}) there is  $p = p(\sigma)\ge 1$  such that for $n \in \N$,
\begin{equation*}
\int_0^t (t - s)^{-\sigma} s^{(n+1)(1-\sigma)-1} \, ds \le \frac{p}{(n+1)^{1 - \sigma}} t^{(n+2)(1-\sigma) -1}, 
\end{equation*}
\begin{equation*}
\int_{t/2}^t (t - s)^{-\sigma} s^{n(1-\sigma) - 1} \, ds \le \frac{p}{(n+1)^{1 - \sigma}} t^{(n+1)(1-\sigma) -1},
\end{equation*}
\begin{equation*}
\int_0^t (t - s)^{-\sigma} s^{n (1-\sigma)} \, ds \le \frac{p}{(n+1)^{1 - \sigma}} t^{(n+1) (1 - \sigma)}.
\end{equation*}
We define $c_1= pc^*\ge c^*$ and $c_2= 2^{(d+\beta)/\alpha} c_1((1-\sigma)^{-1}+p)> c_1 $.

We will prove (\ref{integralqn1}), (\ref{integralqn2}), (\ref{qnestimate1}) simultaneously by induction. They are true for $n = 0$ by  (\ref{global1}, \ref{q0int1}, \ref{q0intx1}) and the choice of $c_1$. Assume that (\ref{integralqn1}), (\ref{integralqn2}), (\ref{qnestimate1}) are true for $n \in \N$, we will show them for $n + 1$. By the definition of $q_{n}(t,x,y)$ and the induction hypothesis we obtain
\begin{eqnarray*}
|q_{n+1}(t,x,y)| &\le& 
c_1 \frac{2^{(d+\beta)/\alpha}}{t^{(d+\beta)/\alpha}} \int_0^{t/2} \int_{\R^d} |q_n(s,z,y)| \, dz \, ds\\
&& + c_1 \frac{c_2^n 2^{(d+\beta)/\alpha}}{(n!)^{1 - \sigma} t^{-1+(d+\beta)/\alpha}} 
\int_{t/2}^t \int_{\R^d} |q_0(t - s,x,z)| \, dz s^{n(1-\sigma) - 1} \, ds\\
&&  \le c_1 \frac{c_1^{n+1} 2^{(d+\beta)/\alpha}}{(n!)^{1-\sigma} t^{(d+\beta)/\alpha}} \int_0^{t/2} s^{(n+1)(1-\sigma) - 1} \, ds\\
&& + c_1 \frac{c_2^n 2^{(d+\beta)/\alpha} c_1}{(n!)^{1-\sigma} t^{-1+(d+\beta)/\alpha}} 
\int_{t/2}^t (t-s)^{(1-\sigma)-1} s^{n(1-\sigma) - 1} \, ds\\
&& \le c_1 \frac{c_2^n t^{(n+1)(1-\sigma)}}{((n+1)!)^{1-\sigma} t^{(d+\beta)/\alpha}} 
\left(c_1 2^{(d+\beta)/\alpha} \frac{1}{1-\sigma}+ c_1 2^{(d+\beta)/\alpha} p \right)\\
&& = c_1 \frac{c_2^{n+1} t^{(n+1) (1 - \sigma)}}{((n+1)!)^{1 - \sigma} t^{(d + \beta)/\alpha}}. 
\end{eqnarray*}
Hence we get (\ref{qnestimate1}) for $n+1$. In particular this gives that the kernel $q_{n+1}(t,x,y)$ is well defined.

By the definition of $q_{n}(t,x,y)$, (\ref{q0int1}) and the induction hypothesis we obtain 
\begin{eqnarray*}\int_{\R^d} |q_{n+1}(t,x,y)| \, dy &\le&   \int_0^{t} \int_{\R^d}\int_{\R^d} |q_0(t - s,x,z)||q_n(s,z,y)| \, dz \, dy \, ds \\ &\le&
 c^* \frac{c_1^{n+1}}{(n!)^{1 - \sigma}} \int_0^{t}(t-s)^{-\sigma}  s^{(n+1)(1-\sigma) - 1}\, ds 
\\ &\le&
 c^* \frac{c_1^{n+1}}{(n!)^{1 - \sigma}} \frac{p}{(n+1)^{1-\sigma}} t^{(n+2)(1-\sigma) - 1}\\ &=&
  \frac{c_1^{n+2}}{((n+1)!)^{1 - \sigma  }}  t^{(n+2)(1-\sigma) - 1},
\end{eqnarray*}            
which proves (\ref{integralqn1}) for $n+1$. Similarly we get (\ref{integralqn2}).

Now we will show (\ref{qnestimate2}). For $n = 0$ this follows from  (\ref{global2}). Assume that (\ref{qnestimate2}) is true for $n \in \N$, we will show it for $n+1$. 

Using our induction hypothesis, (\ref{integralqn1}) and (\ref{integralqn2}) we get for $|x - y| \ge n+2$
\begin{eqnarray*}
 |q_{n+1}(t,x,y)| &=&
\left|\int_0^t \int_{|x-z| \ge \frac{|x-y|}{n+2}} q_0(t-s,x,z) q_n(s,z,y) \, dz \, ds\right. \\
&& + \left. \int_0^t \int_{|x-z| \le \frac{|x-y|}{n+2}} q_0(t-s,x,z) q_n(s,z,y) \, dz \, ds\right| \\
&& \le c_1 e^{-\frac{\lambda |x - y|}{n+2}} \int_0^t \int_{\R^d} |q_n(s,z,y)| \, dz \, ds\\
&& + c_1 \frac{c_2^n}{(n!)^{1-\sigma}} e^{-\frac{\lambda |x - y|}{n+2}}
\int_0^t \int_{\R^d} |q_0(t-s,x,z)| \, dz s^{n (1-\sigma)} \, ds\\
&& \le \frac{c_1}{1-\sigma} \frac{c_1^{n+1} t^{(n + 1)(1-\sigma)}}{((n + 1)!)^{1-\sigma}} e^{-\frac{\lambda |x - y|}{n+2}} 
+ c_1 \frac{c_2^n c^* t^{(n+1)(1-\sigma)} p}{((n+1)!)^{1 - \sigma}} e^{-\frac{\lambda |x - y|}{n+2}}\\
&& = \left(\frac{c_1}{1-\sigma}  c_1^{n+1} +c_2^n c_1^2\right) \frac{ t^{(n + 1) (1-\sigma)}}{((n + 1)!)^{1 - \sigma}} 
e^{-\frac{\lambda |x - y|}{n+2}},
\end{eqnarray*}
which proves  (\ref{qnestimate2}) for $n+1$ since by the choice of constants $\frac{c_1}{1-\sigma}  c_1^{n+1} +c_2^n c_1^2\le c_1c_2^{n+1}$.
\end{proof}

By standard estimates one easily gets that for any $C>0$,
\begin{equation}
\label{sumsigma}
\sum_{n=k}^{\infty} \frac{C^n}{(n!)^{(1 - \sigma)}} \le \frac{C^k}{(k!)^{(1 - \sigma)}}
\sum_{n=k}^{\infty} \frac{C^{n-k}}{((n-k)!)^{(1 - \sigma)}} \le C_1 e^{-k}, \quad k \in \N,
\end{equation}
where $C_1$ depends on $C$ and $\sigma$.

\begin{proposition}
\label{qestimate}
For any $t \in (0,\infty)$, $x, y \in \R^d$ the kernel $q(t,x,y)$ is well defined. For any $t \in (0,\tau]$, $x,y \in \R^d$ we have
$$
|q(t,x,y)| \le \frac{c}{t^{(d+\beta)/\alpha}} e^{-c_3 \sqrt{|x-y|}} \le \frac{c}{t^{(d+\beta)/\alpha} (1 + |x - y|)^{d+1}}.
$$
There exists $a > 0$ ($a$ depends on 
$\tau, \alpha, \beta, \underline{C}, \overline{C}, d, \eta_1, \eta_2, \eta_3, \eta_4, \nu_0, \eps_0, \delta_0$) such that for any $t \in (0,\tau]$, $x,y \in \R^d$, $|x-y| \ge a$ we have
$$
|q(t,x,y)| \le c e^{-c_3 \sqrt{|x-y|}}.
$$
For any $t \in (0,\tau]$ and $x \in \R^d$ we have
\begin{equation}
\label{integralq1}
\int_{\R^d} |q(t,x,y)| \, dy \le c t^{-\sigma},
\end{equation}
\begin{equation}
\label{integralq2}
\int_{\R^d} |q(t,y,x)| \, dy \le c t^{-\sigma}.
\end{equation}
\end{proposition}
\begin{proof}
By (\ref{qnestimate1}) we clearly get $\sum_{n = 0}^{\infty} \left| q_n(t,x,y) \right| \le c t^{-(d+\beta)/\alpha}$. 
Hence, the kernel $q(t,x,y)$ is well defined and  $|q(t,x,y)| \le c t^{-(d+\beta)/\alpha}$. 

For $|x - y| \ge 1$, by (\ref{qnestimate1}), (\ref{qnestimate2}) and (\ref{sumsigma}), we get
\begin{eqnarray*}
 |q(t,x,y)| &=&
\left|\sum_{n = 0}^{\left[\sqrt{|x-y|}-1\right]} q_n(t,x,y) + \sum_{n = \left[\sqrt{|x-y|}\right]}^{\infty} q_n(t,x,y)\right| \\
&& \le
c_1 \sum_{n = 0}^{\left[\sqrt{|x-y|}-1\right]} \frac{c_2^n \tau^{n (1-\sigma)}}{(n!)^{1 - \sigma}} e^{-\lambda \sqrt{|x-y|}} +
c_1 \sum_{n = \left[\sqrt{|x-y|}\right]}^{\infty} \frac{c_2^n \tau^{n (1 - \sigma)}}{(n!)^{1 - \sigma} t^{(d+\beta)/\alpha}}\\
&& \le \frac{c}{t^{(d+\beta)/\alpha}} e^{-c_3 \sqrt{|x-y|}},
\end{eqnarray*}
where $[z]$ denotes the integer part of $z$. Take the smallest $n_0 \in \N$ such that \newline $n_0 (1-\sigma) - 1 \ge (d+\beta)/\alpha$ and $a = n_0^2$. For $\sqrt{|x-y|} \ge \sqrt{a} = n_0$ we get
\begin{eqnarray*}
|q(t,x,y)| &\le&
c_1 \sum_{n = 0}^{\left[\sqrt{|x-y|}-1\right]} \frac{c_2^n \tau^{n (1-\sigma)}}{(n!)^{1 - \sigma}} e^{-\lambda \sqrt{|x-y|}} +
c_1 \sum_{n = \left[\sqrt{|x-y|}\right]}^{\infty} \frac{c_2^n t^{n (1 - \sigma)}}{(n!)^{1 - \sigma} t^{(d+\beta)/\alpha}}\\
&\le& c e^{-c_3 \sqrt{|x-y|}}.
\end{eqnarray*}

The inequalities (\ref{integralq1}) and (\ref{integralq2}) follows easily from (\ref{integralqn1}) and (\ref{integralqn2}).
\end{proof}

The above result enable us to obtain estimates of $u(t,x,y)$.
\begin{corollary}
\label{uintegral}
For any $t \in (0,\infty)$, $x, y \in \R^d$ the kernel $u(t,x,y)$ is well defined. For any $t \in (0,\tau]$, $x,y \in \R^d$ we have
\begin{equation}
\label{utxy}
|u(t,x,y)| \le \frac{c}{t^{-1+(d+\beta)/\alpha}} e^{-c_1 \sqrt{|x-y|}} \le \frac{c}{t^{-1+(d+\beta)/\alpha} (1 + |x - y|)^{d+1}}.
\end{equation}
There exists $a > \eps > 0$ ($a$ depends on $\tau, \alpha, \beta, \underline{C}, \overline{C}, d, \eta_1, \eta_2, \eta_3, \eta_4, \nu_0, \eps_0, \delta_0$) such that for any $t \in (0,\tau]$, $x,y \in \R^d$, $|x-y| \ge a$ we have
$$
|u(t,x,y)| \le c e^{-c_2 \sqrt{|x-y|}}.
$$
For any $t \in (0,\tau]$ and $x \in \R^d$ we have
\begin{equation}
\label{integralu1}
\int_{\R^d} |u(t,x,y)| \, dy \le c,
\end{equation}
\begin{equation}
\label{integralu2}
\int_{\R^d} |u(t,y,x)| \, dy \le c.
\end{equation}
\end{corollary}

\begin{proof} 
We start by obtaining estimates of 
$$I(t,x,y)=\int_0^t \int_{\R^d} p_z(t-s,x-z)|q(s,z,y)| \, dz \, ds.$$
By Lemma \ref{estimate_pytx}, for $0<s<t/2$, we have 
$$p_z(t-s,x-z)\le \frac c{t^{d/\alpha}},$$
and, by Proposition \ref{qestimate}, for $t/2<s<t$,
$$
|q(s,z,y)| \le \frac{c}{t^{(d+\beta)/\alpha}}.
$$
Hence, 
\begin{eqnarray}
\nonumber
 I(t,x,y)
&=&\int_0^{t/2} \int_{\R^d}p_z(t-s,x-z)|q(s,z,y)| \, dz \, ds\\
&&   + \int_{t/2}^t \int_{\R^d} p_z(t-s,x-z)|q(s,z,y)| \, dz \, ds \nonumber\\
&\le&\frac c{t^{d/\alpha}}\int_0^{t/2} \int_{\R^d} |q(s,z,y)| \, dz \, ds 
+ \frac{c}{t^{(d+\beta)/\alpha}}\int_{t/2}^t \int_{\R^d} p_z(t-s,x-z) \, dz \, ds \nonumber\\
&\le&\frac{c}{t^{-1+\sigma +d/\alpha}} + \frac c{t^{-1 + (d+\beta)/\alpha}} \le  \frac c{t^{-1 + (d+\beta)/\alpha}}, \label{ubound1}
\end{eqnarray}
where   (\ref{integralq2}) and  Proposition \ref{pyintegral} were applied to estimate the integrals with respect to the space variable. 

Let $a$ be the constant found in Proposition \ref{qestimate}.
Assume that  $|x-y|\ge 2+2a$. By Lemma
\ref{estimate_pytx}, for $0<s<t$, we have 
$$p_z(t-s,x-z)\le  c  e^{-c_1|x-y|},\quad |x-z|> |x-y|/2>1.$$
Proposition \ref{qestimate} implies that  for $0<s<t$,
$$
|q(s,z,y)| \le  c e^{-c_1\sqrt{|x-y|}}, \quad |y-z|> |x-y|/2>a.
$$
Hence,
\begin{eqnarray}
 I(t,x,y)
&\le&\int_0^{t} \int_{|x-z|> |x-y|/2}\dots  \, dz \, ds + \int_0^{t} \int_{|y-z|> |x-y|/2}\dots  \, dz \, ds\nonumber\\
&\le& c e^{-c_1|x-y|}\int_0^{t} \int_{\R^d} |q(s,z,y)| \, dz \, ds \nonumber \\&+& c e^{-c_1\sqrt{|x-y|}}\int_{0}^t \int_{\R^d} p_z(t-s,x-z) \, dz \, ds\nonumber\\
&\le& c{t^{1 - \sigma}} e^{-c_1|x-y|}+ ct e^{-c_1\sqrt{|x-y|}}\nonumber\\
&\le& c  e^{-c_1\sqrt{|x-y|}}\label{ubound2}.\end{eqnarray}
Combining (\ref{ubound1}) and (\ref{ubound2}) we obtain 
\begin{equation}
\label{Itxy}
 I(t,x,y) \le \frac{c}{t^{-1 + (d+\beta)/\alpha}}  e^{-c_1\sqrt{|x-y|}}.
\end{equation}
By (\ref{defu}), (\ref{htht1}) and (\ref{Itxy}) we get (\ref{utxy}).

Next, (\ref{integralu1}) and (\ref{integralu2}) immediately follow from (\ref{integralq1}), (\ref{integralq2}) and  Proposition \ref{pyintegral}.
\end{proof}

For any $x, v \in \R^d$ put 
$$
D(x,v) = [x_1 - |v_1|, x_1 + |v_1|] \times \ldots \times [x_d - |v_d|, x_d + |v_d|].
$$
\begin{lemma}
\label{double_difference}
For any $t \in (0,\tau]$, $x, y, v \in \R^d$ we have
$$
|p_y(t, x + v) + p_y(t, x - v) - 2p_y(t,x)| \le
c |1 \wedge (t^{-2/\alpha} |v|^2)| \max_{\xi \in D(x,v)} r_y(t,\xi).
$$
\end{lemma}
\begin{proof}
Put $\mathcal{a}_i = g_{i,t}(b_i(y)(x+v))$, $\mathcal{b}_i = g_{i,t}(b_i(y)(x-v))$, $\mathcal{c}_i = \mathcal{d}_i = g_{i,t}(b_i(y)x)$.

Note that
$$
\frac{1}{\det(B(y))} \left(p_y(t, x + v) + p_y(t, x - v) - 2p_y(t,x)\right) = 
\prod_{i=1}^{d} \mathcal{a}_i + \prod_{i=1}^{d} \mathcal{b}_i  - \prod_{i=1}^{d} \mathcal{c}_i - \prod_{i=1}^{d} \mathcal{d}_i. 
$$
For any $j,k \in \{1,\ldots,d\}$, by (\ref{gdelta3}) and (\ref{gdelta2}), we get
$$
\left|\mathcal{a}_j-\mathcal{c}_j - (\mathcal{d}_j - \mathcal{b}_j)\right| \le \frac{c |v|^2}{t^{2/\alpha}}
\max_{z \in [b_j(y)x - |b_j(y)v|, b_j(y)x + |b_j(y)v|]} \tilde{g}_{j,t}(z).
$$
and
\begin{eqnarray*}
&& \left|\mathcal{a}_k-\mathcal{b}_k\right| \left|(\mathcal{d}_j - \mathcal{b}_j)\right| \\
&& \le \frac{c |v|^2}{t^{2/\alpha}}
\max_{z \in [b_k(y)x - |b_k(y)v|, b_k(y)x + |b_k(y)v|]} \tilde{g}_{k,t}(z) 
\max_{z \in [b_j(y)x - |b_j(y)v|, b_j(y)x + |b_j(y)v|]} \tilde{g}_{j,t}(z).
\end{eqnarray*}
Using this, Lemma \ref{difference} and (\ref{gdelta1}) we obtain
$$
|p_y(t, x + v) + p_y(t, x - v) - 2p_y(t,x)| \le \frac{c |v|^2}{t^{2/\alpha}} \max_{\xi \in D(x,v)} r_y(t,\xi).
$$
\end{proof}

\begin{lemma}
\label{Lxpyestimate} 
For any $\xi \in (0,1]$, $\zeta > 0$, $x, y, v \in \R^d$ and $t \in (\xi,\tau+\xi]$ we have
\begin{eqnarray}
\nonumber
&&\sum_{i = 1}^d   \int_{\R} \left|p_y(t,x-y + a_{i}(v) w) + p_y(t,x-y - a_{i}(v) w)- 2 p_y(t,x-y)\right| \, \mu_i(w) \, dw\\
\label{est1}
&&\le c(\xi) e^{-c |x - y|},\\
\nonumber
&&\sum_{i = 1}^d   \int_{|w| \le \zeta} \left|p_y(t,x-y + a_{i}(v) w) + p_y(t,x-y - a_{i}(v) w) - 2 p_y(t,x-y)\right| \, \mu_i(w) \, dw\\
\label{est2}
&&\le c(\xi) \zeta^{2 - \beta}.
\end{eqnarray}
where $c(\xi)$ is a constant depending on $\xi, \tau, \alpha, \beta, \underline{C}, \overline{C}, d, \eta_1, \eta_2, \eta_3, \eta_4, \nu_0, \eps_0, \delta_0$.
\end{lemma}
\begin{proof}
We estimate the summand  corresponding to  $i = 1$.  By Lemma \ref{double_difference},  we get for $w \in \R$,
\begin{eqnarray*}
&&\left|p_y(t,x-y + a_{1}(v) w) + p_y(t,x-y - a_{1}(v) w)- 2p_y(t,x-y)\right|\\
&& \le c t^{-2/\alpha} |w|^2 \max_{\xi \in D(x - y,a_1(v)w)} r_y(t,\xi). 
\end{eqnarray*}
Recall that $\mu_1(w) = 0$ for $|w| \ge 2 \delta$, so we may assume that $|w| \le 2 \delta$. By Lemma \ref{estimate_pytx} we get
$$
\max_{w \in [-2\delta,2\delta]} \max_{\xi \in D(x - y,a_1(v)w)} r_y(t,\xi)\le c_1 t^{-d/\alpha} e^{-c |x-y|}.
$$
Now (\ref{est1}) and (\ref{est2}) follow by the fact that $\mu_1(w) \le c 1_{[-2\delta,2\delta]}(w) |w|^{-1-\beta}$. The last inequality is a consequence of (\ref{nu-ubound}).
\end{proof}

\begin{lemma}
\label{pyeps_limit}
Let $\tau_2 > \tau_1 > 0$ and assume that a function $f_t(x)$ is bounded and uniformly continuous on $[\tau_1,\tau_2] \times \R^d$. Then 
$$
\sup_{t \in [\tau_1,\tau_2], \, x \in \R^d} \left|\int_{\R^d} p_y(\eps_1,x-y) f_t(y) \, dy - f_t(x)\right| \to 0 \quad \text{as} \quad \eps_1 \to 0^+.
$$
\end{lemma}
\begin{proof}
The lemma follows easily from Propostion  \ref{pyintegral}.
\end{proof}

For any $t > 0$, $x,y \in \R^d$ we define
$$
\varphi_y(t,x) = \int_0^t \int_{\R^d} p_z(t-s,x-z) q(s,z,y) \, dz  \, ds.
$$
Clearly we have 
$$
u(t,x,y) = p_y(t,x-y) + \varphi_y(t,x).
$$
For any $t > 0$, $x,y \in \R^d$, $f \in \Bb$ we define
$$
\Phi_t f(x) = \int_{\R^d} \varphi_y(t,x) f(y) \, dy,
$$
$$
U_t f(x) = \int_{\R^d} u(t,x,y) f(y) \, dy,
$$
$$
Q_t f(x) = \int_{\R^d} q(t,x,y) f(y) \, dy.
$$

Now, following the ideas from \cite{KK2018}, we will define the so-called approximate solutions. 
For any $t \ge 0$, $\xi \in [0,1]$, $t + \xi > 0$, $x,y \in \R^d$ we define
$$
\varphi_y^{(\xi)}(t,x) = \int_0^t \int_{\R^d} p_z(t-s+\xi,x-z) q(s,z,y) \, dz  \, ds
$$
and
$$
u^{(\xi)}(t,x,y) = p_y(t+\xi,x-y) + \varphi_y^{(\xi)}(t,x).
$$
For any $t \ge 0$, $\xi \in [0,1]$, $t + \xi > 0$, $x \in \R^d$, $f \in \Bb$ we define
$$
\Phi_t^{(\xi)} f(x) = \int_{\R^d} \varphi_y^{(\xi)}(t,x) f(y) \, dy,
$$
$$
U_t^{(\xi)} f(x) = \int_{\R^d} u^{(\xi)}(t,x,y) f(y) \, dy,
$$
$$
\Phi_0 f(x) = 0, \quad U_0^{(0)} f(x) = U_0 f(x) = f(x).
$$

By the same arguments as in the proof of Corollary \ref{uintegral} we obtain the following result.
\begin{corollary}
\label{ueps_estimate}
For any $t \in [0,\infty)$, $\xi \in [0,1]$, $t + \xi > 0$, $x, y \in \R^d$ the kernel $u^{(\xi)}(t,x,y)$ is well defined. For any $t \in (0,\tau]$, $\xi \in [0,1]$, $x,y \in \R^d$ we have
$$
|u^{(\xi)}(t,x,y)| \le \frac{c}{(t + \xi)^{-1 + (d + \beta)/\alpha} (1 + |x - y|)^{d+1}}.
$$
For any $t \in (0,\tau]$, $\xi \in [0,1]$ and $x \in \R^d$ we have
\begin{equation*}
\int_{\R^d} |u^{(\xi)}(t,x,y)| \, dy \le c,
\end{equation*}
\begin{equation*}
\int_{\R^d} |u^{(\xi)}(t,y,x)| \, dy \le c.
\end{equation*}
\end{corollary}

For any $\zeta > 0$ and $x, y \in \R^d$ we put
$$
\calL_{\zeta} f(x) = 
\sum_{i = 1}^d   \int_{|w| > \zeta} \left[f(x + a_i(x) w) + f(x - a_i(x) w) - 2f(x)\right] \, \mu_i(w) \, dw,
$$
$$
\calL_{\zeta}^y f(x) = 
\sum_{i = 1}^d  \int_{|w| > \zeta} \left[f(x + a_i(y) w) + f(x - a_i(y) w) - 2f(x)\right] \, \mu_i(w) \, dw.
$$

\begin{lemma}
\label{qtf_uniform}
Let $f \in C_0(\R^d)$ and $ \tau_2 > \tau_1 > 0$. Then $Q_t f(x)$ as a function of $(t,x)$ is uniformly continuous on $[\tau_1,\tau_2] \times \R^d$. We have $\lim_{|x| \to \infty} Q_t f(x) = 0$ uniformly in $t \in [\tau_1,\tau_2]$. For each $t > 0$ we have $Q_t f \in C_0(\R^d)$.
\end{lemma}
\begin{proof}
For any $\zeta > 0$, $y \in \R^d$, by Lemma \ref{pycontinuity}, we obtain that
$$
(t,x) \to \calL_{\zeta}^{x}p_y(t,\cdot)(x-y) - \calL_{\zeta}^{y} p_y(t,\cdot)(x-y)
$$
is continuous on $(0,\infty) \times \R^d$. Using this and (\ref{est2}) we show that 
\begin{equation}
\label{q0_cont}
(t,x) \to q_0(t,x,y) \quad \text{is continuous on $(0,\infty) \times \R^d$.}
\end{equation}
By Proposition \ref{integralq02} we have
\begin{equation}
\label{q0_est}
|q_0(t,x,y)| \le \frac{c}{t^{(d + \beta)/\alpha}} e^{-c_1|x-y|}.
\end{equation}

For any $n \in \N$, $t > 0$, $x \in \R^d$ denote
$$
Q_{n,t} f(x) = \int_{\R^d} q_n(t,x,y) f(y) \, dy.
$$
By (\ref{q0_cont}), (\ref{q0_est}) and the dominated convergence theorem we obtain that $(t,x) \to Q_{0,t} f(x)$ is continuous on $(0,\infty) \times \R^d$. By Lemma \ref{integralqn}, for any $t \in (0,\tau]$, $x \in \R^d$, $n \in \N$, we have
\begin{equation}
\label{qnt}
|Q_{n,t} f (x)| \le \frac{c_1^{n+1} t^{(n+1)(1-\sigma) - 1}}{(n!)^{1 - \sigma}} \|f\|_{\infty}.
\end{equation}
Note that for any $t > 0$, $x \in \R^d$, $n \in \N$, $n \ge 1$ we have
$$
Q_{n,t} f (x) = \int_0^t \int_{\R^d} q_0(t-s,x,z) Q_{n-1,s}f(z) \, dz \, ds.
$$
For any $\eps_1 \in (0,\tau_1/2)$, using (\ref{q0_cont}), (\ref{q0_est}) and (\ref{qnt}), we show that
$$
(t,x) \to \int_0^{t - \eps_1} \int_{\R^d} q_0(t-s,x,z) Q_{n-1,s}f(z) \, dz \, ds
$$
is continuous on $[\tau_1,\tau_2] \times \R^d$. Note also that for any $\eps_1 \in (0,\tau_1/2)$, $t \in [\tau_1,\tau_2]$, $x \in \R^d$, $n \in \N$, $n \ge 1$ we have, by (\ref{q0int}),
\begin{eqnarray*}
\left| \int_{t - \eps_1}^t \int_{\R^d} q_0(t-s,x,z) Q_{n-1,s}f(z) \, dz \, ds \right| 
&\le& c \|f\|_{\infty} \int_{t - \eps_1}^t (t-s)^{-\sigma} s^{-\sigma} \, ds\\ 
&\le& c  \tau_1^{-\sigma} \eps_1^{1 - \sigma} \|f\|_{\infty}.
\end{eqnarray*}
This implies that $(t,x) \to Q_{n,t}f(x)$ is continuous on $[\tau_1,\tau_2] \times \R^d$. Using this and (\ref{qnt}) we obtain that $(t,x) \to Q_t f(x) = \sum_{n =0}^{\infty} Q_{n,t}f(x)$ is continuous on $[\tau_1,\tau_2] \times \R^d$. By Proposition \ref{qestimate} we obtain that $\lim_{|x| \to \infty} Q_t f(x) = 0$ uniformly in $t \in [\tau_1,\tau_2]$. This implies the assertion of the lemma.
\end{proof}

The proofs of the next few results are very similar to the proofs of related results in \cite{KRS2018}. We will not repeat these reasonings but we refer the reader to the appropriate proofs in \cite{KRS2018}. 
\begin{proposition}
\label{utfholder}
Choose $\gamma \in (0,\alpha) \cap (0,1]$. For any $t \in (0,\tau]$, $x, x' \in \R^d$, $f \in \Bb$ we have
$$
|U_tf(x) - U_tf(x')| \le c t^{-\gamma/\alpha} |x - x'|^{\gamma} \|f\|_{\infty}.
$$
\end{proposition}
The above result follows by the same arguments as in the proof of \cite[Proposition 3.18]{KRS2018}.

Note that, by Lemma \ref{Lxpyestimate}, for any $\xi \in (0,1]$, $t \in [\xi,\tau+\xi]$, $x, z \in \R^d$ we have
\begin{equation}
\label{tderivative}
\left|\frac{\partial p_z(t,x-z)}{t}\right| = \left|\calL^z p_z(t,\cdot)(x-z)\right|\le c(\xi) e^{-c|x-z|},  
\end{equation}
where $c(\xi)$ is a constant depending on $\xi, \tau, \alpha, \beta, \underline{C}, \overline{C}, d, \eta_1, \eta_2, \eta_3, \eta_4, \nu_0, \eps_0, \delta_0$.

\begin{lemma}
\label{regularity_uteps}
(i) For every $f \in C_0(\R^d)$, $\xi \in (0,1]$ the function $U_t^{(\xi)} f(x)$ belongs to $C^1((0,\infty))$ as a function of $t$ and to $C_0^2(\R^d)$ as a function of $x$. Moreover, 
\begin{equation}
\label{tderivative_estimate}
\left|\frac{\partial}{\partial t} (U_t^{(\xi)} f)(x)\right| \le c({\xi}) \|f\|_{\infty},
\end{equation}
for each $f \in C_0(\R^d)$, $t \in (0,\tau]$, $x \in \R^d$, $\xi \in (0,1]$, where $c(\xi)$ depends on $\xi, \tau, \alpha, \beta, \underline{C}, \overline{C}, d, \eta_1, \eta_2, \eta_3, \eta_4, \nu_0, \eps_0, \delta_0$,

(ii) For every $f \in C_0(\R^d)$ we have
\begin{equation*}
\lim_{t,\xi \to 0^+}\|U_t^{(\xi)} f - f\|_{\infty} = 0.
\end{equation*}

(iii) For every $f \in C_0(\R^d)$ we have
\begin{equation*}
U_t^{(\xi)} f(x) \to 0, \quad \text{as} \quad |x| \to \infty,
\end{equation*}
uniformly in $t \in [0,\tau]$, $\xi \in [0,1]$.

(iv) For every $f \in C_0(\R^d)$ we have
\begin{equation*}
\|U_t^{(\xi)} f - U_t f\|_{\infty} \to 0, \quad \text{as} \quad \xi \to 0^+,
\end{equation*}
uniformly in $t \in [0,\tau]$. 
\end{lemma}
The proof of the above lemma is almost the same as the proof of \cite[Lemma 3.19]{KRS2018}.

For any $t > 0$, $\xi \in (0,1]$, $x \in \R^d$ we put 
$$
\Lambda_t^{(\xi)} f(x) = \frac{\partial}{\partial t} (U_t^{(\xi)} f)(x) - \calL (U_t^{(\xi)} f)(x).
$$
Heuristically, the next lemma states that, if  $\xi$ is small, then $\Lambda_t^{(\xi)} f(x)$ is small. The proof of this lemma almost exactly follows the lines of the proof of \cite[Lemma 3.22]{KRS2018}.
\begin{lemma} $\Lambda_t^{(\xi)} f(x)$ is well defined for every $f \in C_0(\R^d)$, $t \in (0,\tau]$, $\xi \in (0,1]$, $x \in \R^d$ and  we have

\label{heat_u}
(i) for any $f \in C_0(\R^d)$, 
\begin{equation*}
\Lambda_t^{(\xi)} f(x) \to 0, \quad \text{as} \quad \xi \to 0^+,
\end{equation*}
uniformly in $(t,x) \in [\tau_1,\tau_2] \times \R^d$ for every $\tau \ge \tau_2 > \tau_1 > 0$,

(ii) for any $f \in C_0(\R^d)$, 
\begin{equation}
\label{Lambda_s}
\int_0^t \Lambda_s^{(\xi)} f(x) \, ds \to 0, \quad \text{as} \quad \xi \to 0^+,
\end{equation}
uniformly in $(t,x) \in (0,\tau] \times \R^d$.
\end{lemma}

The next result (positive maximum principle) is based on the ideas from \cite[Section 4.2]{KK2018}. Its proof is very similar to the proof of \cite[Lemma 4.3]{KK2018} and it is omitted.
\begin{lemma}
\label{max_principle}
Let us consider the function $v: [0,\infty) \times \R^d \to \R$ and the family of functions $v^{(\xi)}: [0,\infty) \times \R^d \to \R$, $\xi \in (0,1]$. Assume that for each $\xi \in (0,1]$ $\sup_{t \in (0,\tau], x \in \R^d} |v^{(\xi)}(t,x)| < \infty$, $v^{(\xi)}$ is $C^1$ in the first variable and $C^2$ in the second variable. We also assume that (for any $\tau > 0$)

(i)
$$
v^{(\xi)}(t,x) \to v(t,x) \quad \text{as} \quad \xi \to 0^+,
$$
uniformly in $t \in [0,\tau]$, $x \in \R^d$;

(ii)
$$
v^{(\xi)}(t,x) \to 0 \quad \text{as} \quad |x| \to \infty,
$$
uniformly in $t \in [0,\tau]$, $\xi \in (0,1]$;

(iii) for any $0 < \tau_1 < \tau_2 \le \tau$,
$$
\frac{\partial}{\partial t} v^{(\xi)}(t,x)
- \calL v^{(\xi)}(t,x) \to 0 \quad \text{as} \quad \xi \to 0^+,
$$
uniformly in $t \in [\tau_1,\tau_2]$, $x \in \R^d$;

(iv)
$$
v^{(\xi)}(t,x) \to v(0,x) \quad \text{as} \quad \xi \to 0^+\,\,\, \text{and}\,\,\, t \to 0^+,
$$
uniformly in $x \in \R^d$;

(v) for any $x \in \R^d$ $v(0,x) \ge 0$.

Then for any $t \ge 0$, $x \in \R^d$ we have $v(t,x) \ge 0$.
\end{lemma}

\begin{proposition}
\label{positivity}
For any $t > 0$, $x \in \R^d$ and $f \in C_0(\R^d)$ such that $f(x) \ge 0$ for all $x \in \R^d$ we have
$U_t f (x) \ge 0$.
\end{proposition}
\begin{proof}
Let $f \in C_0(\R^d)$ be such that $f(x) \ge 0$ for all $x \in \R^d$. For $t \ge 0$, $x \in \R^d$, $\xi \in (0,1]$ put $v(t,x) = U_t f(x)$, $v^{(\xi)}(t,x) = U_t^{(\xi)} f(x)$. By Lemmas \ref{regularity_uteps} and \ref{heat_u} we obtain that $v(t,x)$, $v^{(\xi)}(t,x)$ satisfy the  assumptions of Lemma \ref{max_principle}. The assertion follows from Lemma \ref{max_principle}.
\end{proof}

\section{Construction and properties of the semigroup of $X_t$}

In this section we will construct the semigroup $T_t$ corresponding to the solution of (\ref{main}). This will be done by, heuristically speaking, adding the impact of long jumps to the semigroup $U_t$, constructed in the last section, corresponding to the solution of (\ref{main}) in which the process $Z$ is replaced by the process with truncated L{\'e}vy measure.  The construction of the semigroup $T_t$ is rather standard. Many arguments in this section are similar to the analogous proofs in \cite{KRS2018}. Such arguments will be omitted. At the end of this section we show that $T_t = P_t$ (where $P_t$ is defined in (\ref{semigroup})) and prove Theorems \ref{mainthm}, \ref{PtL1Linfty} and Proposition \ref{heatkernel}. 

Note that by (\ref{nu-ubound}) we have $\nu_i(x) \le c |x|^{-1-\alpha}$ for $|x| \ge \delta$. Let us introduce the following notation
$$
\lambda_0 = \sum_{i = 1}^d \int_{\R} \left(\nu_i(x) - \mu_i(x)\right) \, dx<\infty.
$$
Note that, by (\ref{KLR}), for any $x\in \R^d$ and  $f \in \Bb$, we have
$$
\calR f(x) =  \sum_{i = 1}^d   \int_{\R} \left(f(x + a_{i}(x) w) - f(x)\right) \left(\nu_i(w) - \mu_i(w)\right) \, dw.
$$ 
We denote, for any $x\in \R^d$ and  $f \in \Bb$, 
$$
\calN f(x) = \sum_{i = 1}^d  \int_{\R} f(x + a_{i}(x) w) \left(\nu_i(w) - \mu_i(w)\right) \, dw. 
$$
It is clear that

\begin{equation}
\label{N_bound}
||\calN f||_\infty \le  \lambda_0|| f||_\infty. 
\end{equation}

For any $t \ge 0$, $\xi \in [0,1]$, $x \in \R^d$ and $n \in \N$,  $f \in \Bb$ we define
\begin{eqnarray}
\label{Psi_nt0}
\Psi_{0,t} f(x) &=& U_t f(x),\\
\label{Psi_nt}
\Psi_{n,t} f(x) &=& \int_0^t  U_{t-s}(\calN(\Psi_{n-1,s}f))(x) \, ds, \quad n \ge 1,\\
\label{Psi_nteps0}
\Psi_{0,t}^{(\xi)} f(x) &=& U_t^{(\xi)} f(x), \\
\label{Psi_nteps}
\Psi_{n,t}^{(\xi)} f(x) &=& \int_0^t  U_{t-s}^{(\xi)}(\calN(\Psi_{n-1,s}^{(\xi)}f))(x) \, ds, \quad n \ge 1.
\end{eqnarray}
We remark that  $\Psi_{n,t}=\Psi_{n,t}^{(0)}$.

For any $x \in \R^d$ we define
\begin{eqnarray*}
T_t f(x) &=& e^{-\lambda_0 t} \sum_{n = 0}^{\infty} \Psi_{n,t} f(x), \quad t \ge 0,\\
T_0 f(x) &=& f(x),\\
T_t^{(\xi)} f(x) &=& e^{-\lambda_0 t} \sum_{n = 0}^{\infty}  \Psi_{n,t}^{(\xi)} f(x), \quad t \ge 0, \,\, \xi \in [0,1].
\end{eqnarray*}

By the same arguments as in Lemma 4.1 and Corollary 4.2 in \cite{KRS2018} one can easily show that $\Psi_{n,t} f(x)$, $\Psi_{n,t}^{(\xi)} f(x)$, $T_t f(x)$ and  $T_t^{(\xi)} f(x)$ are well defined for any $t \ge 0$,  $f \in \Bb$, $x\in \R^d$, $n \in \N$ and $\xi \in [0, 1]$. Moreover, for $t \in [0,\tau]$, $f \in \Bb$, $x\in \R^d$, $\xi \in [0, 1]$ we have $\max\{|T_t f(x)|,|T_t^{(\xi)} f(x)|\} \le c \|f\|_{\infty}$.

Next, we present two regularity results concerning the operators $T_t$. The proofs of these two following results are almost the same as the proofs of Theorems 4.3 and 4.4 in \cite{KRS2018} and are omitted.
\begin{theorem} 
\label{TtL1Linfty} 
For any $\gamma \in (0,\alpha/(d+\beta-\alpha))$, $t \in (0,\tau]$, $x \in \R^d$ and $f \in L^1(\R^d) \cap L^{\infty}(\R^d)$ we have
\begin{equation*}
|T_t f(x)| \le c t^{-\gamma (d+\beta-\alpha)/\alpha}  \|f\|_\infty^{1-\gamma}  \|f\|_1^{\gamma}.
\end{equation*}
\end{theorem}

\begin{theorem}
\label{Holdermain}
Choose $\gamma \in (0,\alpha) \cap (0,1]$. For any $t \in (0,\tau]$, $x, x' \in \R^d$, $f \in \Bb$ we have
$$
|T_tf(x) - T_tf(x')| \le c t^{-\gamma/\alpha} |x - x'|^{\gamma} \|f\|_{\infty}.
$$
\end{theorem}

We need the following auxiliary result. Its proof  is similar to the proof of Lemma 4.10 in \cite{KRS2018} and it is omitted.
\begin{lemma}
\label{distant_support}
Assume that $f \in \Bb$. For any $\eps_1 > 0$ there exists $r \ge 1$ (depending on $\eps_1, \tau, \alpha, \beta, \underline{C}, \overline{C}, d, \eta_1, \eta_2, \eta_3, \eta_4, \nu_0, \eps_0, \delta_0$), such that for any $\xi \in [0,1]$, $t \in [0,\tau]$, $x \in \R^d$, if $\dist(x, \supp(f)) \ge r$, then 
$|T_t^{(\xi)} f(x)| \le \sum_{n = 0}^{\infty} |\Psi_{n,t}^{(\xi)} f(x)| \le \eps_1 \|f\|_{\infty}$.
\end{lemma}

Now we need the following result which, roughly speaking, gives that locally $T_t f$ for $f \in \Bb$ may be approximated by a sequence $T_t f_k$, $k \in \N$, where $f_k \in C_0(\R^d)$.
\begin{proposition}
\label{borelbounded}
For each $t \in (0,\tau]$, $f \in \Bb$ and $R \ge 1$ there exists a sequence $f_k \in C_0(\R^d)$, $k \in \N$ such that $\lim_{k \to \infty} f_k(x) = f(x)$ for almost all $x \in B(0,R)$; for any $k \in \N$ we have $\|f_k\|_{\infty} \le \|f\|_{\infty}$ and for any $x \in B(0,R)$ we have $\lim_{k \to \infty} T_t f_k(x) = T_t f(x)$.
\end{proposition}
\begin{proof}
Fix $t \in (0,\tau]$, $f \in \Bb$, $R \ge 1$ and $k \in \N$, $k \ge 1$. By Lemma \ref{distant_support} there exists 
$R_k \ge R$ such that for any $x \in B(0,R)$ we have
\begin{equation}
\label{large}
|T_t(f 1_{B^c(0,R_k)})(x)| \le \frac{1}{k}.
\end{equation}
Put $g_{1,k}(x) = 1_{B(0,R_k)}(x) f(x)$, $g_{2,k}(x) = 1_{B^c(0,R_k)}(x) f(x)$. By standard arguments there exists $f_k \in C_0(\R^d)$ such that 
$$
\|f_k - g_{1,k}\|_{1} \le \frac{1}{k}.
$$
and $\supp(f_k) \subset B(0,R_k + 1)$, $\|f_k\|_{\infty} \le \|f\|_{\infty}$. 
By Theorem \ref{TtL1Linfty}, for any $x \in \R^d$, we have
$$
|T_t(f_k - g_{1,k})(x)| \le 
\frac{c \|f\|_{\infty}^{1-\alpha/(2d+2\beta- 2\alpha)}}{k^{\alpha/(2d+2\beta- 2\alpha)} t^{1/2}}.
$$
This and (\ref{large}) imply that for any $x \in B(0,R)$ we have $\lim_{k \to \infty} T_t f_k(x) = T_t f(x)$. We also have $\|f_k 1_{B(0,R)} - f 1_{B(0,R)}\|_1 \le 1/k$. Hence, there exists a subsequence $k_m$ such that $\lim_{m \to \infty} f_{k_m}(x) = f(x)$ for almost all $x \in B(0,R)$.
\end{proof}

The next result, Proposition \ref{martingaleproblem} is a very important one, it will be a main tool (in the proof of Theorem  \ref{mainthm}) to show that for any $t > 0$ we have $T_t = P_t$, where $P_t$ is given by (\ref{semigroup}). The steps leading to prove this proposition are very similar to the arguments used in \cite{KRS2018} to show Proposition 4.21. In that paper one shows that $T_t f$ for $f \in C_0(\R^d)$ satisfies the appropriate heat equation in the approximate setting, see \cite[Lemma 4.18]{KRS2018}. Then in the proof of \cite[Proposition 4.21]{KRS2018} one uses this heat equation and the positive maximum principle (for the operator $\calK$), which is formulated in \cite[Lemma 4.19]{KRS2018}. These arguments can be repeated, almost without changes to obtain the proof of Proposition \ref{martingaleproblem}. We decided not to repeat these arguments, since the interested reader can easily find them in \cite{KRS2018}.

\begin{proposition}
\label{martingaleproblem}
For any $t \in (0,\infty)$, $x \in \R^d$ and $f \in C_0^2(\R^d)$ we have
\begin{equation}
\label{martingaleproblem1}
T_t f(x) = f(x) + \int_0^t T_s (\calK f)(x) \, ds.
\end{equation}
\end{proposition}

The next result, Theorem \ref{FellerT} shows that $\{T_t\}$ is a Feller semigroup. Its proof  is almost the same as the proof of \cite[Theorem 4.22]{KRS2018}. Again, we decided not to repeat it.
\begin{theorem}
\label{FellerT}
We have

(i) $T_t: C_0(\R^d) \to C_0(\R^d)$ for any $ t \in (0,\infty)$,

(ii) $T_t f (x) \ge 0$ for any $t > 0$, $x \in \R^d$ and $f \in C_0(\R^d)$ such that $f(x) \ge 0$ for all $x \in \R^d$,

(iii) $T_t 1_{\R^d}(x) = 1$ for any $t > 0$, $x \in \R^d$,

(iv) $T_{t + s} f(x) = T_t(T_s f)(x)$ for any $s, t > 0$, $x \in \R^d$, $f \in C_0(\R^d)$,

(v) $\lim_{t \to 0^+} ||T_t f - f||_{\infty} = 0$ for any $f \in C_0(\R^d)$.

(vi) there exists a nonnegative function $p(t,x,y)$ in $(t,x,y) \in (0,\infty)\times\Rd\times\Rd$; for each fixed $t > 0$, $x \in \R^d$ the function $y \to p(t,x,y)$ is Lebesgue measurable, $\int_{\R^d} p(t,x,y) \, dy = 1$ and $T_t f(x) = \int_{\R^d} p(t,x,y) f(y) \, dy$ for $f \in C_0(\R^d)$.
\end{theorem}

We are now in a position to provide the proofs of Theorems  \ref{mainthm} and \ref{PtL1Linfty}.
\begin{proof}[proof of Theorem \ref{mainthm}] 
From Theorem \ref{FellerT}  we conclude that there is a Feller process $\tilde{X}_t$ with the semigroup $T_t$ on $C_0(\R^d)$. 
Let  $\tilde{\p}^x, \tilde{\E}^x $ be the distribution and expectation of the process $\tilde{X}_t$ starting from $x\in \R^d$. 

By Theorem \ref{FellerT} (vi), Proposition \ref{borelbounded} and Lemma \ref{distant_support} we get
\begin{equation}
\label{semigroup_t}
\tilde{\E}^x f(\tilde{X}_t) = T_t f(x) = \int_{\R^d} p(t,x,y) f(y) \, dy, \quad f \in \Bb, \, t > 0, \, x \in \R^d.
\end{equation}
By Proposition \ref{martingaleproblem}, for any function $f \in C_c^2(\R^d)$, 
the process
\begin{equation*} \label{martingale0}
M_t^{\tilde{X},f}=f(\tilde{X}_t)-f(\tilde{X}_0)- \int_0^t \calK f(\tilde{X}_s)ds 
\end{equation*}
is a  $(\tilde{\p}^x, \calF_t)$ martingale, where $\calF_t  $ is a natural filtration. That is  $\tilde{\p}^x$ solves the martingale problem for $(\calK, C_c^2(\R^d))$. On the other hand, by standard arguments, the unique solution $X$ to the stochastic equation
(\ref{main}) has the law which is the solution to the  martingale problem for 
$(\calK, C_c^2(\R^d))$ (see e.g. \cite[page 120]{K2011}). 

By the Lipschitz property of $a_{i,j}(x)$ and by the Yamada-Watanabe theorem (see \cite[Theorems 37.5 and 37.6]{M1982}) the equation (\ref{main}) has the weak uniqueness property. By this and \cite[Corollary 2.5]{K2011} weak uniqueness holds for the martingale problem for 
$(\calK, C_c^2(\R^d))$. 

Hence $\tilde{X}$ and $X$ have the same law so for any $t > 0$, $x \in \R^d$  and any Borel set $D \subset \R^d$ we have
\begin{equation}
\label{ExAX}
\E^x(1_D(X_t)) = \int_{D} p(t,x,y) \, dy.
\end{equation}
Using this, (\ref{semigroup}) and (\ref{semigroup_t}) we obtain
\begin{equation}
\label{pttt}
P_t f(x) = T_t f(x), \quad t > 0, \, x \in \R^d, \, f \in \Bb.
\end{equation}
Now the assertion of Theorem \ref{mainthm} follows from Theorem \ref{Holdermain} and (\ref{pttt}).
\end{proof}

\begin{proof}[proof of Theorem \ref{PtL1Linfty}] 
The result follows from Theorem \ref{TtL1Linfty} and (\ref{pttt}).
\end{proof}

\begin{proof}[proof of Proposition \ref{heatkernel}]
From  (\ref{ExAX}) we know that transition densities $p(t,x,y)$ for $X_t$ exists. By Lemma \ref{pycontinuity} $(t,x,y) \to p_y(t,x)$ is continuous on $(0,\infty)\times \R^d \times \R^d$. By (\ref{q0formula}) and (\ref{est2}) we obtain that $(t,x,y) \to q_0(t,x,y)$ is continuous on $(0,\infty)\times \R^d \times \R^d$. It follows that $(t,x,y) \to q(t,x,y)$ and $(t,x,y) \to u(t,x,y)$ are continuous on $(0,\infty)\times \R^d \times \R^d$. Using this and Proposition \ref{positivity} we obtain that for any $t> 0$, $x, y \in \R^d$ we have $u(t,x,y) \ge 0$. Denote $u_0(t,x,y) = u(t,x,y)$.

For $n \in \N$, $n \ge 1$, $t > 0$, $x, y \in \R^d$ let us define by induction 
$$
u_n(t,x,y) = \sum_{i= 1}^d \int_0^t \int_{\R^d} u_0(t-s,x,z) \int_{\R} u_{n-1}(s,z+a_i(z)w,y) 
(\nu_i(w) -\mu_i(w)) \, dw \, dz \, ds.
$$
By (\ref{Psi_nt}) we have
$$
\Psi_{n,t} f(x) = \int_{\R^d} u_n(t,x,y) f(y) \, dy.
$$
It follows that for any $t> 0$, $x \in \R^d$ we have 
$p(t,x,y) = e^{-\lambda_0 t} \sum_{n = 0}^{\infty} u_n(t,x,y)$ for almost all $y \in \R^d$ with respect to the Lebesgue measure. Denote $\theta_i(w) = \nu_i(w) -\mu_i(w)$ and $\tilde{p}(t,x,y) = e^{-\lambda_0 t} \sum_{n = 0}^{\infty} u_n(t,x,y)$. For any $t > 0$, $x, y \in \R^d$, $k \in \N$ put $u_0^{(k)}(t)(t,x,y) = u_0(t,x,y) \wedge k$. For $n \in \N$, $n \ge 1$, $k \in \N$, $t > 0$, $x, y \in \R^d$ let us define by induction 
$$
u_n^{(k)}(t,x,y) = \left( \sum_{i= 1}^d \int_0^t \int_{\R^d} u_0^{(k)}(t-s,x,z) \int_{\R} u_{n-1}^{(k)}(s,z+a_i(z)w,y) \theta_i(w) \, dw \, dz \, ds\right) \wedge k.
$$
It follows that $(t,x,y) \to u_n^{(k)}(t,x,y)$ are continuous on $(0,\infty)\times \R^d \times \R^d$ . Clearly, for any $t> 0$, $x, y \in \R^d$ we have $u_n^{(k)}(t,x,y) \ge 0$. We also have $\lim_{k \to \infty} u_n^{(k)}(t,x,y) = u_{n}(t,x,y)$. Hence $\tilde{p}(t,x,y) = \lim_{k \to \infty} e^{-\lambda_0 t} \sum_{n = 0}^{\infty} u_n^{(k)}(t,x,y)$. Therefore $(t,x,y) \to \tilde{p}(t,x,y)$ is lower semi-continuous on $(0,\infty)\times \R^d \times \R^d$ and for any Borel set $D \subset \R^d$ we have $\p^x(X_t \in D) = \int_{D} \tilde{p}(t,x,y) \, dy$.
\end{proof}

\textbf{ Acknowledgements.} 
We thank prof.  J. Zabczyk for communicating to us the problem of the strong Feller property for solutions of SDEs driven by  $\alpha$-stable processes with independent coordinates. We also thank A. Kulik for discussions on the problem treated in the paper.

\end{document}